\documentclass[10pt]{article}
\usepackage[utf8]{inputenc}
\usepackage[T1]{fontenc}             
\usepackage{lmodern}
\usepackage[french,english]{babel}
\usepackage{amsmath}
\usepackage{amsthm}
\usepackage{amssymb}
\usepackage{mathrsfs}
\usepackage{geometry}
\usepackage{graphicx}
\usepackage{comment}
\usepackage{url}
\usepackage{titlesec}
\usepackage[all]{xy}
\usepackage{enumitem}
\usepackage{amsfonts,amsxtra}
\usepackage{stmaryrd} 
\usepackage{tikz}

\usepackage[colorlinks=true]{hyperref}

\newtheorem{df}{Definition}[section]
\newtheorem{rem}[df]{Remark}
\newtheorem{ex}[df]{Example}
\newtheorem{thm}[df]{Theorem}
\newtheorem{pp}[df]{Proposition} 
\newtheorem{lm}[df]{Lemma}

\newtheorem{cor}[df]{Corollary}

\newtheorem*{notation}{Notation}
\newtheorem{thm intro}{Theorem}

\def\so{\mathfrak{so}}

\def\gg{\mathfrak{g}}
\def\hh{\mathfrak{h}}

\def\rr{\mathfrak{r}}

\def\gt{\tilde{\gg}}

\def\gl{\mathfrak{gl}}

\title{\Large{\textbf{Cubic Dirac operators and the strange Freudenthal-de Vries formula for colour Lie algebras}}}
\author{Philippe Meyer\\
\normalsize	Mathematical Institute, University of
Oxford, Oxford, Oxfordshire OX2 6GG, UK\\
\normalsize \texttt{philippe.meyer@maths.ox.ac.uk}}
\date{}

\begin{document}
\maketitle

\begin{center}
\textbf{Abstract}
\end{center}
The aim of this paper is to define cubic Dirac operators for colour Lie algebras. We give a necessary and sufficient condition to construct a colour Lie algebra from an $\epsilon$-orthogonal representation of an $\epsilon$-quadratic colour Lie algebra. This is used to prove a strange Freudenthal-de Vries formula for basic colour Lie algebras as well as a Parthasarathy formula for cubic Dirac operators of colour Lie algebras. We calculate the cohomology induced by this Dirac operator, analogously to the algebraic Vogan conjecture proved by Huang and Pand\v{z}i\'c.

\let\thefootnote\relax
\let\thefootnote\relax
\let\thefootnote\relax\footnote{This work was supported by the Engineering and Physical Sciences Research Council grant [EP/N033922/1].}
\section{Introduction}
The use of Dirac operators in representation theory started by Parthasarathy in \cite{Parthasarathy72} has been developed and generalised in many directions. An important advancement of this approach is the construction of discrete series representations of semisimple Lie groups as kernels of Dirac operators by Atiyah and Schmid (\cite{AtiyahSchmid77}). In \cite{Kos99}, Kostant studies Dirac operators in a more general setting and shows that the introduction of a cubic term in the usual Dirac operator is necessary for non-symmetric pairs. Huang and Pand\v{z}i\'c calculated the cohomology associated to Dirac operators as conjectured by Vogan (\cite{HuangPandzic02}) and then Dirac cohomologies have been studied in various situations (e.g. \cite{Kostant00},\cite{HPR06},\cite{MZ14}). The theory of Dirac operators has been extended to other algebraic structures: for example this has been studied for Lie superalgebras in \cite{Pengpan99},\cite{HuangPandzic05},\cite{KMP14},\cite{Xiao17} and for Hecke algebras and rational Cherednik algebras in \cite{BCT12},\cite{C16},\cite{CdM17}.
\vspace{0.2cm}

Colour Lie algebras generalise both Lie algebras and Lie superalgebras and have been originally introduced by Ree in \cite{Ree60}. The terminology ``colour'' comes from their use in mathematical physics by Rittenberg and Wyler in \cite{RittenbergWyler78-2},\cite{RittenbergWyler78-1} and their study has been intensified by Scheunert in \cite{Scheunert79},\cite{Scheunert83-2},\cite{Scheunert83}.
\vspace{0.2cm}

In this article we define and study cubic Dirac operators for colour Lie algebras. We first recall general properties of the multilinear algebra of vector spaces graded by an abelian group $\Gamma$ with respect to a commutation factor $\epsilon$ of $\Gamma$ (Section \ref{section multilinear}). In particular, we give properties of the generalised Clifford theory developed in \cite{Nishiyama90} and \cite{ChenKang16} and we study some features of roots of ``basic'' colour Lie algebras similar to properties of semisimple complex Lie algebras.
\vspace{0.2cm}

In Section \ref{section colour Lie type and FdV} we give a necessary and sufficient condition to extend the action of an $\epsilon$-quadratic colour Lie algebra $\gg$ on an $\epsilon$-orthogonal representation $V$ to define a colour Lie algebra structure on $\gt=\gg\oplus V$ (see Theorem \ref{thm N(mu+phi)=0}). If this condition is satisfied we say that the representation $V$ is of colour Lie type. This is similar to the conditions of \cite{meyer2019kostant} and generalises the results of Kostant (\cite{Kos99},\cite{Kostant01}) for orthogonal and symplectic complex representations of quadratic Lie algebras and the results of Chen and Kang (\cite{ChenKang2015}) for orthosymplectic complex representations of quadratic Lie superalgebras. If $\gg$ is a basic colour Lie algebra, then this condition for the adjoint representation of $\gg$ compared to the action of the Casimir element of $\gg$ on an oscillator representation of the $\epsilon$-Clifford algebra of $V$ (\cite{Nishiyama90}) leads to the strange Freudenthal-de Vries formula for colour Lie algebras:

\begin{thm intro}
Let $k$ be a field of characteristic $0$. Let $\gg$ be a finite-dimensional basic colour Lie algebra in the sense of Subsection \ref{subsection split CLA} and such that its $\epsilon$-Killing form $B_{\gg}$ is non-degenerate. We have
$$dim_{\epsilon}(\gg)=24B_{\gg^*}(\rho,\rho),$$
where $\rho$ is the Weyl vector of $\gg$ with respect to a Cartan subalgebra $\hh$ of $\gg$ and a choice of a full set of positive roots for the action of $ad(\hh)$.
\end{thm intro}
Note that the proof of Freudenthal and de Vries of this formula for Lie algebras uses the Weyl character formula (\cite{FreudenthalDeVries69}), the proof of Möseneder, Kac and Papi of this formula for Lie superalgebras uses vertex algebras (\cite{KMP14}), but our proof uses Clifford theory and is more in the spirit of Fegan and Steer (\cite{FS89}).
\vspace{0.2cm}

Afterwards, in Section \ref{section cubic Dirac operator}, we define a cubic Dirac operator for the $\epsilon$-orthogonal representation of colour Lie type $V$ of the $\epsilon$-quadratic colour Lie algebra $\gg$. The operator $D$ is defined to be the identity $Id\in {\rm End}(V)$ under the canonical isomorphism between ${\rm End}(V)$ and $V\otimes V$ and the cubic Dirac operator $\tilde{D}$ is defined by adding to $D$ the cubic term $\psi$ in the $\epsilon$-Clifford algebra of $V$ which is defined by the projection on $V$ of the bracket of $\gt=\gg\oplus V$ restricted to $V\times V$ similarly to \cite{Kos99}. We calculate the square of this cubic Dirac operator analogously to Parthasarathy's formula (\cite{Parthasarathy72}) for the square of the Dirac operator of a Cartan decomposition of a reductive Lie algebra:

\begin{thm intro}
Let $k$ be a field of characteristic not $2$ or $3$. Let $\pi:\gg\rightarrow\so_{\epsilon}(V,(\phantom{v},\phantom{v}))$ be a finite-dimensional $\epsilon$-orthogonal representation of colour Lie type of a finite-dimensional $\epsilon$-quadratic colour Lie algebra $(\gg,B_{\gg})$. Let $\gt=\gg\oplus V$ be the associated colour Lie algebra and let $\tilde{D}$ be the cubic Dirac operator. In $U_{\epsilon}(\gt)\otimes C_{\epsilon}(V,(\phantom{v},\phantom{v}))$, we have
$$\tilde{D}^2=\Omega(\gt)\otimes 1-\Delta(\Omega(\gg))+\frac{1}{24}\Big(Tr_{\epsilon}(ad_{\gt}(\Omega(\gt)))-Tr_{\epsilon}(ad_{\gg}(\Omega(\gg)))\Big) 1\otimes 1,$$
where $\Omega(\gt)$ (resp. $\Omega(\gg)$) is the Casimir element of $\gt$ (resp. $\gg$).
\end{thm intro}
\vspace{0.2cm}

In this theorem, $\Delta$ is the diagonal embedding of $\gg$ in $U_{\epsilon}(\gt)\otimes C_{\epsilon}(V,(\phantom{v},\phantom{v}))$ coming from the fact that the $\epsilon$-orthogonal colour Lie algebra $\so_{\epsilon}(V,(\phantom{v},\phantom{v}))$ is canonically isomorphic to the second $\epsilon$-exterior power $\Lambda_{\epsilon}^2(V)$ and that we can quantise $\Lambda_{\epsilon}^2(V)$ to the $\epsilon$-Clifford algebra of $V$ (Section \ref{section multilinear}).
\vspace{0.2cm}

Finally, this cubic Dirac operator induces a differential map and we calculate the cohomology of this differential complex in Section \ref{section cohomology}, similarly to the proof of Huang and Pand\v{z}i\'c of the Vogan conjecture (\cite{HuangPandzic02}):

\begin{thm intro}
Let $k$ be a field of characteristic $0$. Let $\pi:\gg\rightarrow\so_{\epsilon}(V,(\phantom{v},\phantom{v}))$ be a finite-dimensional $\epsilon$-orthogonal representation of colour Lie type of a finite-dimensional $\epsilon$-quadratic colour Lie algebra $(\gg,B_{\gg})$. Let $\gt=\gg\oplus V$ be the associated colour Lie algebra, let $\tilde{D}$ be the cubic Dirac operator and let $d$ be its induced differential on $(U_{\epsilon}(\gt)\otimes C_{\epsilon}(V,(\phantom{v},\phantom{v})))^{\gg}$. Then we have
$$Ker(d)=\mathcal{Z}_{\epsilon}(\Delta(\gg))\oplus Im(d).$$
\end{thm intro}

To prove this theorem, we have to define a ``standard resolution'' of the field $k$ similar to the Koszul and de Rham differentials and to prove the analogues of the Koszul cohomology and the Poincaré lemma in the context of $\epsilon$-symmetric modules (Lemma \ref{lm koszul complex}).
\vspace{0.2cm}

\begin{notation}
Let $k$ be a field of characteristic not $2$ or $3$, let $\Gamma$ be an abelian group and let $\epsilon$ be a commutation factor of $\Gamma$. The reader which is only interested in supersymmetry and Lie superalgebras can assume $\Gamma=\mathbb{Z}_2$ and $\epsilon(a,b)=(-1)^{ab}$ for all $a,b\in\mathbb{Z}_2$.
\end{notation}

\section{Multilinear algebra over vector spaces graded by an abelian group}\label{section multilinear}

\subsection{Basic definitions and properties about graded vector spaces}

Let $V=\bigoplus \limits_{\gamma \in \Gamma} V_{\gamma}$ be a $\Gamma$-graded vector space. For $v \in V_{\gamma}$ we set $|v|:=\gamma$. For convenience, whenever the degree of an element is used in a formula, it is assumed that this element is homogeneous and that we extend by linearity the formula for non-homogeneous elements. For $v,w$ in $V$ we denote $\epsilon(|v|,|w|)$ by $\epsilon(v,w)$ and $\epsilon(|v|,|v|)$ by $\epsilon(v)$. We always assume that a basis of a $\Gamma$-graded vector space is composed of homogeneous elements and that the base field $k$ has the trivial $\Gamma$-gradation.
\vspace{0.2cm}

If $V$ and $W$ are $\Gamma$-graded vector spaces, then $V\oplus W$ and $V\otimes W$ are $\Gamma$-graded in an obvious way. Define the $\Gamma$-graded vector space ${\rm Homgr}(V,W)$ by
$${\rm Homgr}(V,W)_{\gamma}:=\lbrace f\in {\rm Hom}(V,W) ~ | ~ f(V_a)\subseteq W_{a+\gamma} \quad \forall a \in \Gamma \rbrace \qquad \forall \gamma\in \Gamma.$$
If $V$ and $W$ are finite-dimensional, then ${\rm Homgr}(V,W)={\rm Hom}(V,W)$ and so $V^*={\rm Hom}(V,k)$ is $\Gamma$-graded. If $\lbrace e_i \rbrace$ is a basis of $V$ and $\lbrace e_i^* \rbrace$ is the basis of $V^*$ dual to $\lbrace e_i \rbrace$ in the sense that $e_i^*(e_j)=\delta_{ij}$, then $|e_i^*|=-|e_i|$. Define also $\mathcal{E}\in {\rm End}(V)$ by $\mathcal{E}(v):=\epsilon (v)v$, the $\epsilon$-trace of $f\in {\rm End}(V)$ by $Tr_{\epsilon}(f):=Tr(\mathcal{E}\circ f)$ and the $\epsilon$-dimension of $V$ by $dim_{\epsilon}(V):=Tr_{\epsilon}(Id_V)$.

\begin{rem} \label{rem VtensW iso WtensV}
The isomorphism of $\Gamma$-graded vector spaces between $V\otimes W$ and $W\otimes V$ is given by $v\otimes w\mapsto \epsilon(v,w)w\otimes v$ and hence the isomorphism of $\Gamma$-graded vector spaces between ${\rm End}(V)$ and $V\otimes V^*$ by $f\mapsto \sum\limits_i f(e_i)\otimes e_i^*$.
\end{rem}

It-is known that if $V$ is a $\Gamma$-graded vector space, there is a unique right group action $\pi : S_n \rightarrow GL(V^{\otimes n})$ of the permutation group $S_n$ on $V^{\otimes n}$ such that the action of a transposition $\tau_{i,i+1} \in S_n$ is given by
\begin{equation*}
\pi(\tau_{i,i+1})(v_1\otimes \ldots \otimes v_n)=-\epsilon(v_i,v_{i+1}) v_1 \otimes \ldots \otimes v_{i+1}\otimes v_i \otimes \ldots \otimes v_n
\end{equation*}
for all $v_1,\ldots,v_n \in V$. For arbitrary elements $\sigma,\sigma' \in S_n$, this action is given by
$$\pi(\sigma)(v_1\otimes \ldots \otimes v_n)=p(\sigma ; v_1,\ldots, v_n) v_{\sigma(1)}\otimes \ldots \otimes v_{\sigma(n)}$$ 
where
$$p(\sigma ; v_1,\ldots, v_n)=sgn(\sigma) \displaystyle \prod_{\substack{1 \leq i < j \leq n \\ \sigma^{-1}(i)>\sigma^{-1}(j) }} \epsilon(v_i,v_j)$$
and satisfies to
\begin{equation*}
p(\sigma\sigma' ; v_1,\ldots, v_n)=p(\sigma' ; v_{\sigma(1)},\ldots, v_{\sigma(n)})p(\sigma ; v_1,\ldots, v_n), \qquad p(Id ; v_1,\ldots,v_n)=1.
\end{equation*}
Recall that for multilinear maps on $\Gamma$-graded vector spaces there is a notion of symmetry and antisymmetry with respect to the commutation factor $\epsilon$.

\begin{df}
Let $V$ and $W$ be $\Gamma$-graded vector spaces and let $B : V\times V \rightarrow W$ be a bilinear map.
\begin{enumerate}[label=\alph*)]
\item We say that $B$ is $\epsilon$-symmetric if $B(v,w)=\epsilon(v,w)B(w,v)$ for all $v,w\in V$.
\item We say that $B$ is $\epsilon$-antisymmetric if $B(v,w)=-\epsilon(v,w)B(w,v)$ for all $v,w\in V$.
\item If $V$ is finite-dimensional, $\lbrace e_i\rbrace$ is a basis of $V$ and $(\phantom{v},\phantom{v})$ is a non-degenerate $\epsilon$-symmetric bilinear form of degree $0$ on $V$, then the unique basis $\lbrace e^i\rbrace$ of $V$ such that $(e_i,e^j)=\delta_{ij}$ is called the dual basis of $\lbrace e_i\rbrace$.
\end{enumerate}
\end{df}

Unless otherwise stated, we always assume the bilinear forms to be of degree $0$. The basic features of a basis and its dual basis are given in the following remark.

\begin{rem} \label{rem Id_v in V tens V}
Let $V$ be a finite-dimensional $\Gamma$-graded vector space together with a non-degenerate $\epsilon$-symmetric bilinear form $(\phantom{v},\phantom{v})$, let $\lbrace e_i\rbrace$ be a basis of $V$ and let $\lbrace e^i \rbrace$ be its dual basis.
\begin{enumerate} [label=\alph*)]
\item We have $|e^i|=-|e_i|$.
\item \label{rem Id_v in V tens V part iso V et V dual} As $\Gamma$-graded vector spaces, $V$ is canonically isomorphic to $V^*$ by $v\mapsto (v,\phantom{v})$ and we have $e_i^*=\epsilon(e_i)(e^i,\phantom{v})$.
\item \label{rem la forme de Id dans V tens V} Using \ref{rem Id_v in V tens V part iso V et V dual} and Remark \ref{rem VtensW iso WtensV} the $\Gamma$-graded vector spaces ${\rm End}(V)$ and $V\otimes V$ are canonically isomorphic and the identity map $Id_V\in {\rm End}(V)$ corresponds to
$$\sum\limits_i e^i\otimes e_i=\sum\limits_i \epsilon(e_i)e_i\otimes e^i \in V\otimes V.$$
In particular, $\sum\limits_i e^i\otimes e_i$ is independent of the choice of the basis $\lbrace e_i\rbrace$.
\item We have
\begin{equation}\label{eq decomp (v,ei)ei}
v=\sum\limits_i(v,e^i)e_i=\sum\limits_i\epsilon(e_i)(v,e_i)e^i=\sum\limits_i(e_i,v)e^i\qquad \forall v\in V.
\end{equation}
\item Let $W$ be a $\Gamma$-graded vector space and let $f : V\otimes V \rightarrow W$ be an $\epsilon$-antisymmetric bilinear map. Using \ref{rem la forme de Id dans V tens V} we have
$$f(\sum\limits_i e^i\otimes e_i)=0.$$
\end{enumerate}
\end{rem}

\subsection{$\epsilon$-exterior algebra and $\epsilon$-Clifford algebra}

Let $(V,(\phantom{v},\phantom{v}))$ be a finite-dimensional $\Gamma$-graded vector space together with a non-degenerate $\epsilon$-symmetric bilinear form. In this subsection we recall basic properties of the $\epsilon$-exterior algebra $V$ and of the $\epsilon$-Clifford algebra of $(V,(\phantom{v},\phantom{v}))$ (see \cite{ChenKang16}).

\begin{df}
\begin{enumerate}[label=\alph*)]
\item The $\epsilon$-Clifford algebra $C_{\epsilon}(V,(\phantom{v},\phantom{v}))$ of $V$ is defined by $$C_{\epsilon}(V,(\phantom{v},\phantom{v})):=T(V)/I(V),$$
where $T(V)$ is the tensor algebra over $V$ and $I(V)$ is the two-sided ideal of $T(V)$ generated by elements of the form
$$x\otimes y+\epsilon(x,y)y\otimes x-2(x,y)1 \qquad \forall x,y \in V.$$
\item The $\epsilon$-exterior algebra $\Lambda_{\epsilon}(V)$ of $V$ is defined to be $C_{\epsilon}(V,(\phantom{v},\phantom{v}))$ where $(\phantom{v},\phantom{v})$ is totally degenerate.
\end{enumerate}
\end{df}

The tensor algebra $T(V)$ is a $\mathbb{Z}\times\Gamma$-graded algebra and then the $\epsilon$-Clifford algebra $C_{\epsilon}(V,(\phantom{v},\phantom{v}))$ is a $\mathbb{Z}_2\times\Gamma$-graded algebra and the $\epsilon$-exterior algebra $\Lambda_{\epsilon}(V)$ is a $\mathbb{Z}\times\Gamma$-graded algebra. For the universal property of the $\epsilon$-Clifford algebra, see \cite{ChenKang16}. The $\epsilon$-Clifford algebra $C_{\epsilon}(V,(\phantom{v},\phantom{v}))$ is filtered and its associated graded algebra is isomorphic to $\Lambda_{\epsilon}(V)$. Define $e: V\rightarrow {\rm End}(\Lambda_{\epsilon}(V))$ by $e(v)(w):=v\wedge w$ for all $v\in V$, for all $w\in \Lambda_{\epsilon}(V)$ and define $i: V\rightarrow {\rm End}(T(V))$ by
$$i(v)(w_1\otimes\ldots\otimes w_n):=\sum\limits_{i}(-1)^{i-1} \epsilon(v,w_1+\ldots +w_{i-1})(v,w_i)w_1\otimes \ldots \otimes \hat{w_i}\otimes \ldots\otimes w_n \quad \forall v\in V, ~\forall w_1\otimes\ldots\otimes w_n\in T(V).$$
This descends to define $i:V\rightarrow {\rm End}(\Lambda_{\epsilon}(V))$ and by the universal property of the $\epsilon$-Clifford algebra, the map $\gamma:V\rightarrow {\rm End}(\Lambda_{\epsilon}(V))$ defined by $\gamma(v)=e(v)+i(v)$ for all $v\in V$ extends to define $\gamma:C_{\epsilon}(V,(\phantom{v},\phantom{v}))\rightarrow {\rm End}(\Lambda_{\epsilon}(V))$. Using the maps $e$ and $i$ we can define oscillator representations for a particular class of $\epsilon$-Clifford algebras.

\begin{df} \label{def oscillator}
Suppose that one of the following holds:
\begin{enumerate}[label=\alph*)]
    \item there are two maximal isotropic $\Gamma$-graded subspaces $W$ and $W'$ such that $V=W\oplus W'$;
    \item \label{def oscillator b} there are two maximal isotropic $\Gamma$-graded subspaces $W$ and $W'$, a one-dimensional subspace $U$ of degree $0$ orthogonal to $W$ and $W'$, $u\in U$ such that $(u,u)=1$ and $V=W\oplus W'\oplus U$.
\end{enumerate}
Define the oscillator representation $m:V\rightarrow {\rm End}(\Lambda_{\epsilon}(W))$ by
$$m(w+w'):=e(w)+2i(w') \qquad \forall w\in W, ~ \forall w'\in W',$$
and, if \ref{def oscillator b} holds, by
$$m(u):=Id|_{\Lambda_{\epsilon}(W)_{\bar{0}}}-Id|_{\Lambda_{\epsilon}(W)_{\bar{1}}},$$
where $\Lambda_{\epsilon}(W)_{\bar{0}}:=\bigoplus \limits_{i\in \mathbb{N}}\Lambda^{2i}_{\epsilon}(W)$ and $\Lambda_{\epsilon}(W)_{\bar{1}}:=\bigoplus \limits_{i\in \mathbb{N}}\Lambda^{2i+1}_{\epsilon}(W)$.
\end{df}

Using the universal property of the $\epsilon$-Clifford algebra, this extends to define $m:C_{\epsilon}(V,(\phantom{v},\phantom{v}))\rightarrow {\rm End}(\Lambda_{\epsilon}(W))$. There is a quantisation map from the $\epsilon$-exterior algebra to the $\epsilon$-Clifford algebra:

\begin{df}\label{df quantisation}
Let $n\in \mathbb{N}$ and suppose that $char(k)=0$ or $n<char(k)$. Define the quantisation map $Q_n :\Lambda^n_{\epsilon}(V) \rightarrow C_{\epsilon}(V,(\phantom{v},\phantom{v}))$ by
$$Q_n(v_1\wedge\ldots \wedge v_n):=\frac{1}{n!}\sum\limits_{\sigma \in S_n}p(\sigma;v_1,\ldots,v_n)v_{\sigma(1)}\cdot\ldots\cdot v_{\sigma(n)}.$$
\end{df}

Since the graded map $gr(Q_n)\in {\rm End}(\Lambda_{\epsilon}^n(V))$ is the identity, the map $Q_n$ is injective. Furthermore, if $char(k)=0$, then the map $\sum\limits_n Q_n:\Lambda_{\epsilon}(V) \rightarrow C_{\epsilon}(V,(\phantom{v},\phantom{v}))$ is an isomorphism of $\Gamma$-graded vector spaces and its inverse map is $\gamma':C_{\epsilon}(V,(\phantom{v},\phantom{v}))\rightarrow \Lambda_{\epsilon}(V)$ given by $\gamma'(v)=\gamma(v)(1)$ for all $v\in C_{\epsilon}(V,(\phantom{v},\phantom{v}))$ (see \cite{ChenKang16}).
\vspace{0.2cm}

We now extend $(\phantom{v},\phantom{v})$ to $\Lambda_\epsilon(V)$. Let $u_1,\ldots,u_n,v_1,\ldots,v_n\in V$ and consider the bilinear forms on $T^n(V)$ defined by
$$B(u_1\otimes\ldots\otimes u_n,v_1\otimes\ldots\otimes v_n):=\prod\limits_{i=0}^{n-1} (u_{n-i},v_{1+i})$$
and
$$(u_1\otimes\ldots\otimes u_n,v_1\otimes\ldots\otimes v_n)_{\Lambda}:=\sum\limits_{\sigma\in S_n}B(\pi(\sigma)(u_1\otimes\ldots\otimes u_n),v_1\otimes\ldots\otimes v_n).$$
One can check that $$(u_1\otimes\ldots\otimes u_n,v_1\otimes\ldots\otimes v_n)_{\Lambda}=\sum\limits_{\sigma\in S_n}B(u_1\otimes\ldots\otimes u_n,\pi(\sigma)(v_1\otimes\ldots\otimes v_n))$$
and so this descends to define a bilinear form on $\Lambda^n_{\epsilon}(V)$.
\begin{pp}\label{pp Lambda dual isom to Lambda}
Let $n\in \mathbb{N}$ and suppose that $char(k)=0$ or $n<char(k)$.
\begin{enumerate}[label=\alph*)]
\item The bilinear form $(\phantom{v},\phantom{v})_{\Lambda}$ of $\Lambda^n_{\epsilon}(V)$ defined above is $\epsilon$-symmetric and non-degenerate.
\item The linear map $\Lambda^n_{\epsilon}(V)\rightarrow \Lambda^n_{\epsilon}(V)^*$ given by $f\mapsto (f,\phantom{f})_{\Lambda}$ is an isomorphism of $\Gamma$-graded vector spaces and for $f\in \Lambda^n_{\epsilon}(V)^*$, the element $\tilde{f}\in\Lambda^n_{\epsilon}(V)$ corresponding to $f$ under this isomorphism satisfies to
\begin{equation*}
\tilde{f}=\frac{1}{n!} \sum\limits_{i_1,\ldots,i_n}f(e^{i_1}\wedge \ldots \wedge e^{i_n})e_{i_n}\wedge \ldots \wedge e_{i_1}\in \Lambda^n_{\epsilon}(V)
\end{equation*}
and
\begin{equation*}
Q_n(\tilde{f})=\frac{1}{n!} \sum\limits_{i_1,\ldots,i_n}f(e^{i_1}\wedge \ldots \wedge e^{i_n})e_{i_n}\cdot \ldots \cdot e_{i_1} \in C_{\epsilon}(V,(\phantom{v},\phantom{v}))
\end{equation*}
where $\lbrace e_i\rbrace$ is a basis of $V$ and $\lbrace e^i \rbrace$ its dual basis.
\end{enumerate}
\end{pp}

\begin{proof}
A straightforward calculation shows that $(\phantom{v},\phantom{v})_{\Lambda}$ is $\epsilon$-symmetric. Let $(i_1',\ldots,i_n')$. It is non-degenerate since 
$$(e_{i_n'}\wedge \ldots \wedge e_{i_1'},e^{i_1'}\wedge \ldots \wedge e^{i_n'})_{\Lambda}=|\lbrace \sigma \in S_n ~ | ~ \sigma\cdot (i_1',\ldots,i_n') =(i_1',\ldots,i_n') \rbrace|.$$
We want to show that
\begin{equation}\label{proof formula Lambdan dual}
\Big(\frac{1}{n!} \sum\limits_{i_1,\ldots,i_n}f(e^{i_1}\wedge \ldots \wedge e^{i_n})e_{i_n}\wedge \ldots \wedge e_{i_1},e^{i_1'}\wedge \ldots \wedge e^{i_n'}\Big)_{\Lambda}=f(e^{i_1'}\wedge \ldots \wedge e^{i_n'}).
\end{equation}
Let $(i_1,\ldots,i_n)$ such that for all $\sigma \in S_n$ we have $\sigma\cdot (i_1,\ldots,i_n)\neq(i_1',\ldots,i_n')$. Hence
$$\Big(e_{i_n}\wedge \ldots \wedge e_{i_1},e^{i_1'}\wedge \ldots \wedge e^{i_n'}\Big)_{\Lambda}=0$$
and so
$$\Big(\frac{1}{n!} \sum\limits_{i_1,\ldots,i_n}f(e^{i_1}\wedge \ldots \wedge e^{i_n})e_{i_n}\wedge \ldots \wedge e_{i_1},e^{i_1'}\wedge \ldots \wedge e^{i_n'}\Big)_{\Lambda}=\frac{1}{n!}\sum\limits_{(i_1,\ldots,i_n)\in S}f(e^{i_1}\wedge \ldots \wedge e^{i_n})(e_{i_n}\wedge \ldots \wedge e_{i_1},e^{i_1'}\wedge \ldots \wedge e^{i_n'})_{\Lambda}$$
where $S$ is the set of all $(i_1,\ldots,i_n)$ such that there exists $\sigma \in S_n$ such that $\sigma\cdot (i_1,\ldots,i_n)=(i_1',\ldots,i_n')$. For each $(i_1,\ldots,i_n)\in S$ we have
$$f(e^{i_1}\wedge \ldots \wedge e^{i_n})e_{i_n}\wedge \ldots \wedge e_{i_1}=f(e^{i_1'}\wedge \ldots \wedge e^{i_n'})e_{i_n'}\wedge \ldots \wedge e_{i_1'}$$
and then
$$\Big(\frac{1}{n!} \sum\limits_{i_1,\ldots,i_n}f(e^{i_1}\wedge \ldots \wedge e^{i_n})e_{i_n}\wedge \ldots \wedge e_{i_1},e^{i_1'}\wedge \ldots \wedge e^{i_n'}\Big)_{\Lambda}=\frac{|S|}{n!}(e_{i_n'}\wedge \ldots \wedge e_{i_1'},e^{i_1'}\wedge \ldots \wedge e^{i_n'})_{\Lambda}.$$
Finally,
$$(e_{i_n'}\wedge \ldots \wedge e_{i_1'},e^{i_1'}\wedge \ldots \wedge e^{i_n'})_{\Lambda}=|\lbrace \sigma \in S_n ~ | ~ \sigma\cdot (i_1',\ldots,i_n') =(i_1',\ldots,i_n') \rbrace|$$
and since
$$|S||\lbrace \sigma \in S_n ~ | ~ \sigma\cdot (i_1',\ldots,i_n') =(i_1',\ldots,i_n') \rbrace|=n!$$
we obtain Equation \eqref{proof formula Lambdan dual}.
\vspace{0.2cm}

For $i_j,i_k$, by \eqref{eq decomp (v,ei)ei}, we have
\begin{align}
\sum\limits_{i_1,\ldots,i_n}f(e^{i_1}\wedge \ldots \wedge e^{i_n})(e_{i_j},e_{i_k})&=\sum\limits_{i_1,\ldots,\hat{i_j},\ldots,i_n}f(e^{i_1}\wedge \ldots \wedge \sum\limits_{i_j}(e_{i_j},e_{i_k})e^{i_j}\wedge \ldots \wedge e^{i_k} \wedge \ldots \wedge e^{i_n})\nonumber\\
&=\sum\limits_{i_1,\ldots,\hat{i_j},\ldots,i_n}f(e^{i_1}\wedge \ldots \wedge e_{i_k}\wedge \ldots \wedge e^{i_k} \wedge \ldots \wedge e^{i_n})\nonumber\\
&=0. \label{proof formula Lambdan dual to Clifford}
\end{align}
Using the Clifford relation $uv=-\epsilon(u,v)vu+2(u,v)$ we have that
$$Q_n(e_{i_n}\wedge \ldots \wedge e_{i_1})=e_{i_n}\cdot \ldots \cdot e_{i_1}+B$$
where $B$ is a sum of elements of the form $(e_{i_j},e_{i_k})X$ and hence by Equation \eqref{proof formula Lambdan dual to Clifford} we obtain
$$Q_n(\tilde{f})=\frac{1}{n!} \sum\limits_{i_1,\ldots,i_n}f(e^{i_1}\wedge \ldots \wedge e^{i_n})e_{i_n}\cdot \ldots \cdot e_{i_1}.$$
\end{proof}

\subsection{Colour Lie algebras}

In this subsection, following \cite{Ree60}, we define colour Lie algebras, representations of colour Lie algebras (\cite{RittenbergWyler78-2}, \cite{RittenbergWyler78-1}, \cite{Scheunert79}) and give some properties of the moment map of an $\epsilon$-orthogonal representation of an $\epsilon$-quadratic colour Lie
algebra.

\begin{df} \label{def CLA}
A colour Lie algebra is a $\Gamma$-graded vector space $\gg$ together with a bilinear map $\lbrace \phantom{x}, \phantom{x} \rbrace : \gg \times \gg \rightarrow \gg $ such that 
\begin{enumerate}[label=\alph*)]
\item $\lbrace \gg_{\alpha}, \gg_{\beta} \rbrace \subseteq \gg_{\alpha+\beta}$ for all $\alpha,\beta \in \Gamma$,
\item $\lbrace x,y \rbrace=-\epsilon(x,y)\lbrace y,x \rbrace$ for all $x,y \in \gg \qquad$ ($\epsilon$-antisymmetry),
\item $\epsilon(z,x)\lbrace x , \lbrace y , z \rbrace \rbrace+\epsilon(x,y)\lbrace y , \lbrace z , x \rbrace \rbrace+\epsilon(y,z)\lbrace z , \lbrace x , y \rbrace \rbrace=0$ for all $x,y,z \in \gg \qquad$ ($\epsilon$-Jacobi identity).
\end{enumerate}
\end{df}

Lie (super)algebras are colour Lie algebras and here are some examples:

\begin{ex}
Let $V$ be a $\Gamma$-graded vector space.
\begin{enumerate}[label=\alph*)]
\item The associative $\Gamma$-graded algebra ${\rm Homgr}(V,V)$ is a colour Lie algebra for the bracket $\lbrace a,b \rbrace := ab-\epsilon(a,b)ba$ for $a,b$ in ${\rm Homgr}(V,V)$ and is denoted $\gl_{\epsilon}(V)$.
\item Let $(\phantom{v},\phantom{v})$ be an $\epsilon$-symmetric bilinear form on $V$. We set $\so_{\epsilon}(V,(\phantom{v},\phantom{v})):=\bigoplus \limits_{\gamma \in \Gamma} \so_{\epsilon}(V,(\phantom{v},\phantom{v}))_{\gamma}$ where
$$\so_{\epsilon}(V,(\phantom{v},\phantom{v}))_{\gamma}:=\lbrace f \in {\rm Homgr}(V,V)_{\gamma} ~ | ~ (f(v),w)+\epsilon(f,v)(v,f(w))=0 \quad \forall v,w \in V \rbrace.$$
One can show that $\so_{\epsilon}(V,(\phantom{v},\phantom{v}))$ is stable under the bracket of $\gl_{\epsilon}(V)$ and hence is a colour Lie algebra.
\end{enumerate}
\end{ex}

Let $\gg$ and $\gg'$ be colour Lie algebras. A degree $0$ linear map $f \in {\rm Homgr}(\gg, \gg')$ is a morphism of colour Lie algebras if $f(\lbrace x,y \rbrace)=\lbrace f(x),f(y) \rbrace$ for all $x,y \in \gg$ and we say that $\gg$ and $\gg'$ are isomorphic if $f$ is a linear isomorphism. A representation $V$ of $\gg$ is a $\Gamma$-graded vector space $V$ together with a morphism of colour Lie algebras $\pi : \gg \rightarrow \gl_{\epsilon}(V)$ and the adjoint representation $ad:\gg\rightarrow \gl_{\epsilon}(\gg)$ is defined by $ad(x):=\lbrace x,\phantom{y}\rbrace$ for all $x\in \gg$. An $\epsilon$-orthogonal representation $V$ of $\gg$ is a $\Gamma$-graded vector space $V$ together with a non-degenerate $\epsilon$-symmetric bilinear form $(\phantom{v},\phantom{v})$ and a morphism of colour Lie algebras $\pi : \gg \rightarrow \so_{\epsilon}(V,(\phantom{v},\phantom{v}))$. We say that $\gg$ is $\epsilon$-quadratic if there is a bilinear form $B$ on $\gg$ which is $\epsilon$-symmetric, non-degenerate and ad-invariant is the sense that
$$B(\lbrace x,y \rbrace , z)=-\epsilon(x,y)B(y,\lbrace x,z \rbrace) \qquad \forall x,y,z \in \gg.$$
The $\epsilon$-Killing form $K$ is the $\epsilon$-symmetric ad-invariant bilinear form on $\gg$ given by
$$K(x,y):=Tr_{\epsilon}(ad(x)ad(y)) \qquad \forall x,y\in \gg.$$

\begin{rem} \label{rem tensor product of reps}
\begin{enumerate}[label=\alph*)]
\item If $\pi_1 : \gg\rightarrow {\rm End}(V_1)$ and $\pi_2 : \gg\rightarrow {\rm End}(V_2)$ are representations of a colour Lie algebra $\gg$, then $\pi:\gg\rightarrow {\rm End}(V_1\otimes V_2)$ given by
$$\pi(x)(v_1\otimes v_2):=\pi_1(x)(v_1)\otimes v_2+\epsilon(x,v_1)v_1\otimes \pi_2(x)(v_2) \qquad \forall x\in \gg, ~ \forall v_1\otimes v_2\in V_1\otimes V_2$$
is also a representation of $\gg$. In particular, if $V$ is a finite-dimensional representation of $\gg$, then $T(V)$ and $\Lambda_{\epsilon}(V)$ are representations of $\gg$. If $(V,(\phantom{v},\phantom{v}))$ is an $\epsilon$-orthogonal representation of $\gg$, then $C_{\epsilon}(V,(\phantom{v},\phantom{v}))$ is also a representation of $\gg$.

\item If $\gg$ is a finite-dimensional $\epsilon$-quadratic colour Lie algebra, the isomorphism between ${\rm End}(\gg)$ and $\gg\otimes\gg$ is $\gg$-equivariant.
\end{enumerate}
\end{rem}

Unless otherwise stated, we suppose all $\epsilon$-orthogonal representations of dimension at least two.

\begin{df} Let $\gg$ be a finite-dimensional colour Lie algebra. 
\begin{enumerate}[label=\alph*)]
\item The $\epsilon$-universal enveloping algebra $U_{\epsilon}(\gg)$ of $\gg$ is defined by $$U_{\epsilon}(\gg):=T(\gg)/I(\gg),$$
where $T(\gg)$ is the tensor algebra over $\gg$ and $I(\gg)$ is the two-sided ideal of $T(\gg)$ generated by elements of the form
$$x\otimes y-\epsilon(x,y)y\otimes x-\lbrace x,y\rbrace \qquad \forall x,y \in \gg.$$
\item The $\epsilon$-symmetric algebra $S_{\epsilon}(\gg)$ of $\gg$ is defined to be $U_{\epsilon}(\gg)$ where the bracket of $\gg$ is trivial.
\item If $\gg$ is $\epsilon$-quadratic, then the Casimir element $\Omega(\gg) \in U_{\epsilon}(\gg)$ is defined by
$$\Omega(\gg):=\sum\limits_i x^ix_i \in U_{\epsilon}(\gg)$$
where $\lbrace x_i\rbrace$ is a basis of $\gg$ and $\lbrace x^i \rbrace$ its dual basis.
\end{enumerate}
\end{df}

The action of $\gg$ on $T(\gg)$ of Remark \ref{rem tensor product of reps} defines an action of $\gg$ on $U_{\epsilon}(\gg)$. For the universal property of the $\epsilon$-universal enveloping algebra, see \cite{Scheunert79}. Since $\Omega(\gg) \in U_{\epsilon}(\gg)$ corresponds to $Id_{\gg}\in {\rm End}(\gg)$, then $\Omega(\gg)$ is in the $\epsilon$-centre $\mathcal{Z}_{\epsilon}(\gg)$ of $U_{\epsilon}(\gg)$ where the $\epsilon$-centre is defined by
$$\mathcal{Z}_{\epsilon}(\gg):=\lbrace x \in U_{\epsilon}(\gg) ~ |  ~ xy=\epsilon(x,y) yx \quad \forall y \in U_{\epsilon}(\gg) \rbrace.$$

We now define and study the moment map associated to an $\epsilon$-orthogonal representation.

\begin{df} \label{df moment map}
Let $\pi : \gg \rightarrow \so_{\epsilon}(V,(\phantom{v},\phantom{v}))$ be a finite-dimensional $\epsilon$-orthogonal representation of a finite-dimensional $\epsilon$-quadratic colour Lie algebra $(\gg,B_{\gg})$. We define the moment map $\mu : V\times V \rightarrow \gg$  to be the bilinear map given by
$$B_{\gg}(x,\mu(v,w))=(\pi(x)(v),w) \qquad \forall v,w \in V, ~ \forall x \in \gg.$$
\end{df}

The moment map $\mu$ is $\epsilon$-antisymmetric and of degree $0$. The classical example of a moment map is the moment map of the fundamental representation of $\so_{\epsilon}(V,(\phantom{v},\phantom{v}))$.

\begin{pp} \label{pp moment map canonique}
Let $V$ be a finite-dimensional $\Gamma$-graded vector space together with a non-degenerate $\epsilon$-symmetric bilinear form $(\phantom{v},\phantom{v})$.
\begin{enumerate}[label=\alph*)]
\item Consider the $\epsilon$-orthogonal representation of the $\epsilon$-quadratic colour Lie algebra $(\so_{\epsilon}(V,(\phantom{v},\phantom{v})),B)$ where $B(f,g)=-\frac{1}{2}Tr_{\epsilon}(fg)$ for all $f,g \in \so_{\epsilon}(V,(\phantom{v},\phantom{v}))$. Then, the corresponding moment map $\mu_{can}$ satisfies
\begin{equation*}
\mu_{can}(u,v)(w)=\epsilon(v,w)(u,w)v-(v,w)u \qquad \forall u,v,w\in V.
\end{equation*}
\item The moment map $\mu_{can}:\Lambda^2_{\epsilon}(V)\rightarrow\so_{\epsilon}(V,(\phantom{v},\phantom{v}))$ is an isomorphism of $\Gamma$-graded vector spaces.
\item For $f\in \so_{\epsilon}(V,(\phantom{v},\phantom{v}))$, we have
\begin{equation*}
\mu_{can}^{-1}(f)=-\frac{1}{2}\sum\limits_{i}f(e^i)\wedge e_i,\qquad Q_2(\mu_{can}^{-1}(f))=-\frac{1}{2}\sum\limits_{i} f(e^i)e_i
\end{equation*}
where $\lbrace e_i\rbrace$ is a basis of $V$ and $\lbrace e^i \rbrace$ its dual basis.
\end{enumerate}
\end{pp}
\begin{proof}
A straightforward calculation shows a) (or see \cite{meyer2019kostant}). For b), see \cite{ChenKang16} if $char(k)=0$ or chapter 3 of \cite{MeyerThesis}. We now prove c). Let $v\in V$. We have
\begin{equation*}
\mu_{can}\Big(-\frac{1}{2}\sum\limits_{i}f(e^i)\wedge e_i\Big)(v)=-\frac{1}{2}\sum\limits_{i,j}\Big(\epsilon(e_i,v)(f(e^i),v)e_i-(e_i,v)f(e^i)\Big)=f(v).
\end{equation*}
Since $\sum\limits_{i}(f(e^i),e_i)=0$ by Remark \ref{rem Id_v in V tens V}, we have
$$Q_2(\mu^{-1}_{can}(f))=-\frac{1}{2}\sum\limits_{i} f(e^i)e_i+\frac{1}{2}\sum\limits_{i}(f(e^i),e_i)=\sum\limits_{i} f(e^i)e_i.$$
\end{proof}

\subsection{Basic colour Lie algebras} \label{subsection split CLA}
Suppose that $char(k)=0$. Let $(\gg,B_{\gg})$ be a finite-dimensional $\epsilon$-quadratic colour Lie algebra and let $\hh$ be a Cartan subalgebra of the Lie algebra $\gg_0$. For $\alpha\in \hh^*$, define
$$\gg^\alpha=\lbrace x\in\gg ~ | ~ \lbrace h,x\rbrace=\alpha(h)x \qquad \forall h\in \hh  \rbrace.$$
We have
$$\lbrace\gg^\alpha,\gg^\beta\rbrace \subseteq \gg^{\alpha+\beta} \qquad \forall \alpha,\beta \in \hh^*.$$
Suppose that $dim(\gg^\alpha)\leq 1$ for all non-zero $\alpha\in \hh^*$, $ad(\hh)$ is diagonalisable in the sense that $\gg= \bigoplus\limits_{\alpha\in \hh^*} \gg^{\alpha}$ and $\hh$ is self-normalising in the sense that $\gg^0=\hh$. A non-zero element $\alpha\in \hh^*$ with $\gg^\alpha\neq \lbrace 0 \rbrace$ is called a root of $\gg$ with respect to $\hh$ and the set of all roots is denoted $\Delta$.

\begin{pp} \label{pp roots spaces orthogonaux}
We have the following.
\begin{enumerate}[label=\alph*)]
\item If $\alpha,\beta\in \hh^*$ are such that $\alpha+\beta\neq0$, then $B_{\gg}(\gg^\alpha,\gg^\beta)=0.$
\item The restriction of $B_{\gg}$ to $\gg^\alpha\times \gg^{-\alpha}$ is non-degenerate.
\end{enumerate}
\end{pp}
\begin{proof}
Since $\alpha+\beta\neq 0$, there exists $h\in\hh$ such that $(\alpha+\beta)(h)\neq 0$. Let $x\in \gg^\alpha$ and let $y\in \gg^\beta$. We have
$$\alpha(h)B_{\gg}(x,y)=B_{\gg}(\lbrace h,x\rbrace,y)=-B_{\gg}(x,\lbrace h,y\rbrace)=-\beta(h)B_{\gg}(x,y),$$
then $(\alpha+\beta)(h)B_{\gg}(x,y)=0$ which implies $B_{\gg}(x,y)=0$. This shows $a)$ and $b)$ follows since $B_{\gg}$ is non-degenerate.
\end{proof}
Thus, if $\alpha$ is a root, $-\alpha$ is also a root and so one can define $\Delta^+$ and $\Delta^-$ such that $\Delta=\Delta^+\cup \Delta^-$ and such that $\alpha\in \Delta^{\pm}$ implies $-\alpha \in \Delta^{\mp}$. Suppose that the vanishing of a linear combination $\sum\limits_{\alpha\in \Delta^{\pm}}c_\alpha\alpha$ with integral coefficients $c_\alpha\geq 0$ implies that $c_\alpha=0$ for all $\alpha\in \Delta^{\pm}$. The elements in $\Delta^+$ (resp. $\Delta^-$) are called the positive (resp. negative) roots. For $\alpha\in \Delta$ we set $\epsilon(\alpha):=\epsilon(x)$ where $x$ is a non-zero element of $\gg^{\alpha}$ and we define the Weyl vector $\rho$ by
$$\rho:=\frac{1}{2}\sum\limits_{\alpha\in\Delta^+}\epsilon(\alpha)\alpha.$$
If we define $\gg^+=\bigoplus\limits_{\alpha\in\Delta^+}\gg^\alpha$ and $\gg^-=\bigoplus\limits_{\alpha\in\Delta^-}\gg^\alpha$, we have that $\gg^+$ and $\gg^-$ are nilpotent colour Lie algebras and a decomposition 
$$\gg=\gg^-\oplus \hh\oplus \gg^+.$$
By Proposition \ref{pp roots spaces orthogonaux} $(\hh,B_{\gg}|_{\hh})$ is a non-degenerate quadratic vector space. For $\alpha\in\hh^*$ define $h_{\alpha}\in \hh$ to be the unique element in $\hh$ such that $\alpha=B_{\gg}(h_{\alpha},\phantom{X})\in \hh^*$ and define a non-degenerate symmetric bilinear form $B_{\gg^*}(\phantom{v},\phantom{v})$ on $\hh^*$ by
$$B_{\gg^*}(\alpha,\beta)=B_{\gg}(h_\alpha,h_\beta) \qquad \forall \alpha,\beta \in \hh^*.$$
\begin{pp} \label{pp inner product between roots}
We have the following.
\begin{enumerate}[label=\alph*)]
\item Let $\lbrace h_i\rbrace$ be a basis of $\hh$ and let $\lbrace h^i\rbrace$ its dual basis. We have
$$B_{\gg^*}(\alpha,\beta)=\sum\limits_i \alpha(h^i)\beta(h_i) \qquad \forall \alpha, \beta \in \hh^*.$$
\item For a root $\alpha$ and $x\in \gg^\alpha,y\in \gg^{-\alpha}$ we have $\lbrace x,y \rbrace=B_{\gg}(x,y)h_{\alpha}$.
\end{enumerate}
\end{pp}

\begin{proof}
$a)$ We have
$$B_{\gg^*}(\alpha,\beta)=B_{\gg}(h_\alpha,h_\beta)=\sum\limits_iB_{\gg}(B_{\gg}(h_\alpha,h^i)h_i,\beta)=\sum\limits_i\alpha(h^i)\beta(h_i).$$
$b)$ For $z\in \gg^\gamma$, where $\gamma\in \hh^*\backslash \lbrace 0 \rbrace$, by Proposition \ref{pp roots spaces orthogonaux} we have
$$B_{\gg}(z,\lbrace x,y \rbrace)=B_{\gg}(z,B_{\gg}(x,y)h_\alpha)=0.$$
For $h\in\hh$ we have
\begin{align*}
    B_{\gg}(h,\lbrace x,y \rbrace)&=B_{\gg}(\lbrace h,x\rbrace,y)\\
    &=\alpha(h)B_{\gg}(x,y)\\
    &=B_{\gg}(h_\alpha,h)B_{\gg}(x,y)\\
    &=B_{\gg}(h,B_{\gg}(x,y)h_\alpha)
\end{align*}
and so $\lbrace x,y \rbrace=B_{\gg}(x,y)h_{\alpha}$.
\end{proof}
\vspace{0.2cm}

\begin{df}
Let $\pi : \gg\rightarrow {\rm End}(V)$ be a representation of $\gg$. We say that $V$ is a representation of $\gg$ of highest weight $\lambda$ if there exist a homogeneous non-zero element $v\in V$ and a linear form $\lambda\in \hh^*$ such that
\begin{itemize}
    \item[$\bullet$] $v$ generates $V$ as $\gg$-module;
    \item[$\bullet$] $v$ is annihilated by $\gg^+$;
    \item[$\bullet$] $\pi(h)(v)=\lambda(h)v$ for all $h\in \hh$.
\end{itemize}
\end{df}

The set of the elements $u\in V$ such that $\pi(h)(u)=\lambda(h)u$ for all $h\in \hh$ is one dimensional. We have:

\begin{pp} \label{pp action Casimir on highest weight}
The Casimir element $\Omega(\gg)\in U_{\epsilon}(\gg)$ satisfies to
$$\pi(\Omega(\gg))=B_{\gg^*}(\lambda+2\rho,\lambda)Id_V.$$
\end{pp}
\begin{proof}
Since $\Omega(\gg)\in \mathcal{Z}_{\epsilon}(\gg)$, cleary it acts on $V$ as a multiple of the identity. Let $\lbrace h_i\rbrace$ be a basis of $\hh$ and let $\lbrace h^i\rbrace$ be its dual basis. For each $\alpha\in\Delta^+$ let $e_\alpha\in \gg^\alpha$ and $e_{-\alpha}\in \gg^{-\alpha}$ be such that $B_{\gg}(e_\alpha,e_{-\alpha})=1$. The set
$$\lbrace h_i \rbrace \cup \lbrace e_\alpha, ~ \forall \alpha\in \Delta^+ \rbrace\cup \lbrace e_{-\alpha}, ~ \forall \alpha\in \Delta^+ \rbrace$$
is a basis of $\gg$ and its dual basis is
$$\lbrace h^i \rbrace \cup \lbrace e_{-\alpha}, ~ \forall \alpha\in \Delta^+ \rbrace\cup \lbrace \epsilon(\alpha)e_\alpha, ~ \forall \alpha\in \Delta^+ \rbrace.$$
Hence we have
$$\Omega(\gg)=\sum\limits_i h^ih_i +\sum\limits_{\alpha\in \Delta^+} e_{-\alpha}e_\alpha+\sum\limits_{\alpha\in \Delta^+} \epsilon(\alpha) e_{\alpha}e_{-\alpha}.$$
By Proposition \ref{pp inner product between roots} we have $\lbrace e_\alpha,e_{-\alpha}\rbrace =h_{\alpha}$ and so
\begin{equation}\label{proof action casimir on highest weight 1}
 \Omega(\gg)=\sum\limits_i h^ih_i +2\sum\limits_{\alpha\in \Delta^+} e_{-\alpha}e_\alpha+\sum\limits_{\alpha\in \Delta^+} \epsilon(\alpha)h_{\alpha}.
\end{equation}
We have
\begin{equation}\label{proof action casimir on highest weight 2}
\pi(\sum\limits_i h^ih_i)(v)=\sum\limits_i \lambda(h^i)\lambda(h_i)v=B_{\gg^*}(\lambda,\lambda)v.
\end{equation}
Since $e_\alpha\in \gg^+$ annihilates $v$, we have
\begin{equation}\label{proof action casimir on highest weight 3}
\pi(2\sum\limits_{\alpha\in \Delta^+} e_{-\alpha}e_\alpha)(v)=0.
\end{equation}
Finally, we have
$$\pi(\sum\limits_{\alpha\in \Delta^+} \epsilon(\alpha)h_{\alpha})(v)=\sum\limits_{\alpha\in \Delta^+} \epsilon(\alpha)\lambda(h_{\alpha})v$$
and since
\begin{align*}
    2B_{\gg^*}(\rho,\lambda)&=\sum\limits_i\sum\limits_{\alpha\in\Delta^+}\epsilon(\alpha)\alpha(h^i)\lambda(h_i)=\sum\limits_i\sum\limits_{\alpha\in\Delta^+}\epsilon(\alpha)B_{\gg}(h_\alpha,h^i)B_{\gg}(h_\lambda,h_i)\\
    &=\sum\limits_{\alpha\in\Delta^+}\epsilon(\alpha)B_{\gg}(h_\lambda,\sum\limits_iB_{\gg}(h_\alpha,h^i)h_i)=\sum\limits_{\alpha\in\Delta^+}\epsilon(\alpha)B_{\gg}(h_\lambda,h_\alpha)\\
    &=\sum\limits_{\alpha\in\Delta^+}\epsilon(\alpha)\lambda(h_\alpha)
\end{align*}
we obtain that
\begin{equation}\label{proof action casimir on highest weight 4}
\pi(\sum\limits_{\alpha\in \Delta^+} \epsilon(\alpha)h_{\alpha})(v)=2B_{\gg^*}(\rho,\lambda)v.
\end{equation}
From Equations \eqref{proof action casimir on highest weight 1}, \eqref{proof action casimir on highest weight 2}, \eqref{proof action casimir on highest weight 3} and \eqref{proof action casimir on highest weight 4} we obtain that
$$\pi(\Omega(\gg))(v)=B_{\gg^*}(\lambda+2\rho,\lambda)v.$$
\end{proof}

\section{Representations of colour Lie type and the strange Freudenthal-de Vries formula}\label{section colour Lie type and FdV}

In this section, from a finite-dimensional $\epsilon$-orthogonal representation $\pi : \gg\rightarrow \so_{\epsilon}(V,(\phantom{v},\phantom{v}))$ of a finite-dimensional $\epsilon$-quadratic colour Lie algebra $\gg$, we show how to construct an $\epsilon$-quadratic colour Lie algebra structure on $\gg\oplus V$ using a slight different point of view than in \cite{meyer2019kostant}. Using this characterisation we state a strange Freudenthal and de Vries formula (\cite{FreudenthalDeVries69}) for basic colour Lie algebras.
\vspace{0.2cm}

The $\epsilon$-Clifford algebra $C_{\epsilon}(V,(\phantom{v},\phantom{v}))$ is $\mathbb{Z}_2\times \Gamma$-graded and so consider the commutation factor $\tilde{\epsilon}$ of $\mathbb{Z}_2\times\Gamma$ given by
$$\tilde{\epsilon}((a_1,\gamma_1),(a_2,\gamma_2)):=(-1)^{a_1a_2}\epsilon(\gamma_1,\gamma_2) \qquad \forall (a_1,\gamma_1),(a_2,\gamma_2)\in \mathbb{Z}_2\times\Gamma.$$
Then $C_{\epsilon}(V,(\phantom{v},\phantom{v}))$ is a $(\mathbb{Z}_2\times\Gamma,\tilde{\epsilon})$-colour Lie algebra for the bracket
$$\lbrace x,y \rbrace:=x\cdot y-\tilde{\epsilon}(x,y)y\cdot x \qquad \forall x,y \in C_{\epsilon}(V,(\phantom{v},\phantom{v})).$$
Define $\pi_* : \gg\rightarrow C_{\epsilon}(V,(\phantom{v},\phantom{v}))$ by 
$$\pi_*=-\frac{1}{2}Q_2\circ \mu_{can}^{-1}\circ \pi,$$
where $Q_2 :\Lambda^2_{\epsilon}(V) \rightarrow C_{\epsilon}(V,(\phantom{v},\phantom{v}))$ is the quantisation map of Definition \ref{df quantisation} and where $\mu_{can}:\Lambda_{\epsilon}^2(V)\rightarrow \so_{\epsilon}(V,(\phantom{v},\phantom{v}))$ is the isomorphism of Proposition \ref{pp moment map canonique}. We have
\begin{equation}\label{eq pi in clifford}
    \pi_*(x)=\frac{1}{4}\sum\limits_{i}\pi(x)(e^i) e_i \qquad \forall x\in \gg
\end{equation}
where $\lbrace e_i\rbrace$ is a basis of $V$ and $\lbrace e^i \rbrace$ its dual basis. The map $\pi_*$ is a colour Lie algebra morphism and so we extend it to $\pi_* : U_{\epsilon}(\gg)\rightarrow C_{\epsilon}(V,(\phantom{v},\phantom{v}))$ using the universal property of the $\epsilon$-universal enveloping algebra $U_{\epsilon}(\gg)$.

\begin{thm} \label{thm N(mu+phi)=0}
Let $\pi : \gg \rightarrow \so_{\epsilon}(V,(\phantom{v} , \phantom{v}))$ be a finite-dimensional $\epsilon$-orthogonal representation of a finite-dimensional $\epsilon$-quadratic colour Lie algebra $(\gg,B_{\gg})$ and let $\mu : \Lambda^2_{\epsilon}(V)\rightarrow \gg$ be its moment map. Let $\phi : \Lambda^2_{\epsilon}(V) \rightarrow V$ be of degree $0$ and satisfy
\begin{align}
\pi(x)(\phi(v,w))&=\phi(\pi(x)(v),w)+\epsilon(x,v)\phi(v,\pi(x)(w)) \quad &\forall x \in \gg, ~ \forall v,w \in V \label{phi g-inv}, \\
(\phi(u,v),w)&=-\epsilon(u,v)(v,\phi(u,w)) \quad &\forall u,v,w \in V \label{phi ()-inv}.
\end{align}
Let $\gt:=\gg\oplus V$, let $B_{\gt}:=B_{\gg}\perp (\phantom{v},\phantom{v})$ and let $\lbrace \phantom{v},\phantom{v} \rbrace \in \Lambda_{\epsilon}^2(\gt)\rightarrow \gt$ be the unique map which extends the bracket of $\gg$, the action of $\gg$ on $V$ and such that
$$\lbrace v,w\rbrace=\mu(v,w)+\phi(v,w) \qquad \forall v,w \in V.$$
Define the cubic term $\psi \in C_{\epsilon}(V,(\phantom{v},\phantom{v}))$ to be the element corresponding to the $\epsilon$-alternating trilinear map $(u,v,w)\mapsto\frac{1}{2}(u,\phi(v,w))$ under the isomorphism of vector spaces between $\Lambda^3_{\epsilon}(V)^*$ and $\Lambda^3_{\epsilon}(V)$ of Proposition \ref{pp Lambda dual isom to Lambda} and under the quantisation map $Q_3 :\Lambda^3_{\epsilon}(V) \rightarrow C_{\epsilon}(V,(\phantom{v},\phantom{v}))$.
\vspace{0.2cm}

\noindent
The following are equivalent:
\begin{enumerate}[label=\alph*)]
\item $(\gt,B_{\gt},\lbrace \phantom{v},\phantom{v} \rbrace)$ is an $\epsilon$-quadratic colour Lie algebra.
\item $\pi_*(\Omega(\gg))+\psi^2\in k.$
\end{enumerate}
\end{thm}
\vspace{0.1cm}

\begin{proof} 
For $u,v,w\in\gt$, if $u,v$ or $w$ is an element of $\gg$ then a straightforward calculation shows that
$$\epsilon(w,u)\lbrace \lbrace u,v \rbrace , w \rbrace+\epsilon(u,v)\lbrace \lbrace v,w \rbrace , u \rbrace +\epsilon(v,w)\lbrace \lbrace w,u \rbrace , v \rbrace=0.$$
Let $u,v,w \in V$ and $x\in \gg$. We have
\begin{align*}
B_{\gg}(\lbrace u , \lbrace v,w \rbrace \rbrace,x)&=-\epsilon(u,v+w)\epsilon(u+v+w,x)(\pi(x)(\phi (v,w)), u)
\end{align*}
and by Equation \eqref{phi g-inv} we obtain
\begin{align*}
B_{\gg}(\lbrace u , \lbrace v,w \rbrace \rbrace,x)&=-\epsilon(u,v+w)B_{\gg}( \lbrace v,\lbrace w,u\rbrace \rbrace,x)+\epsilon(u,v+w)\epsilon(v,w)B_{\gg}(\lbrace w,\lbrace v,u\rbrace \rbrace,x).
\end{align*}
Hence,
$$B_{\gg}(\epsilon(w,u)\lbrace \lbrace u,v \rbrace , w \rbrace+\epsilon(u,v)\lbrace \lbrace v,w \rbrace , u \rbrace +\epsilon(v,w)\lbrace \lbrace w,u \rbrace , v \rbrace,x)=0.$$
It follows that $\gt$ is a colour Lie algebra if and only if
$$(\epsilon(v_3,v_1)\lbrace \lbrace v_1,v_2 \rbrace , v_3 \rbrace+\epsilon(v_1,v_2)\lbrace \lbrace v_2,v_3 \rbrace , v_1 \rbrace +\epsilon(v_2,v_3)\lbrace \lbrace v_3,v_1 \rbrace , v_2 \rbrace,v_4)=0 \quad \forall v_1,v_2,v_3,v_4 \in V.$$
Define the map $J\in \Lambda_{\epsilon}(V)^*$ by
\begin{align*}
J(v_1,v_2,v_3,v_4)&=\epsilon(v_1,v_3)(\epsilon(v_3,v_1)\lbrace \lbrace v_1,v_2 \rbrace,v_3\rbrace+\epsilon(v_1,v_2)\lbrace \lbrace v_2,v_3\rbrace , v_1\rbrace+\epsilon(v_2,v_3)\lbrace \lbrace v_3,v_1\rbrace,v_2\rbrace,v_4) ~ \forall v_1,v_2,v_3,v_4\in V
\end{align*}
and $\tilde{J}\in \Lambda_{\epsilon}(V)$ to be the corresponding element of $J$ under the isomorphism of Proposition \ref{pp Lambda dual isom to Lambda}. By Proposition \ref{pp Lambda dual isom to Lambda}, we have
\begin{align*}
Q_4(\tilde{J})&=\frac{1}{4!}\sum\limits_{i,j,n,m}\tilde{J}(e^i\wedge e^j\wedge e^n\wedge e^m)e_me_ne_je_i
\end{align*}
and so
$$Q_4(\tilde{J})=\frac{1}{4!}\sum\limits_{i,j,n,m}\Big(\lbrace \lbrace e^i,e^j \rbrace,e^n\rbrace+\epsilon(e_i,e_j+e_n)\lbrace \lbrace e^j,e^n\rbrace , e^i\rbrace+\epsilon(e_i+e_j,e_n)\lbrace \lbrace e^n,e^i\rbrace,e^j\rbrace,e^m\Big)e_me_ne_je_i.$$
Since
$$e_ne_je_i=\epsilon(e_j+e_n,e_i)e_ie_ne_j-2\epsilon(e_j,e_i)(e_n,e_i)e_j+2(e_j,e_i)e_n,$$
we have
\begin{align*}
&\sum\limits_{i,j,n,m}\epsilon(e_i,e_j+e_n)(\lbrace \lbrace e^j,e^n\rbrace , e^i\rbrace,e^m)e_me_ne_je_i\\
&=\sum\limits_{i,j,n,m}(\lbrace \lbrace e^i,e^j \rbrace,e^n\rbrace,e^m)e_me_ne_je_i-4\sum\limits_{j,n}(\mu(e^j,e^n)(e_n)+\phi(\phi(e^j,e^n),e_n))e_j
\end{align*}
and since $e_ne_j=-\epsilon(e_n,e_j)e_je_n+2(e_n,e_j)$ we also have
\begin{align*}
&\sum\limits_{i,j,n,m}\epsilon(e_i+e_j,e_n)(\lbrace \lbrace e^n,e^i\rbrace,e^j\rbrace,e^m)e_me_ne_je_i=-\sum\limits_{i,j,n,m}\epsilon(e_j,e_n)(\lbrace \lbrace e^i,e^n\rbrace,e^j\rbrace,e^m)e_me_ne_je_i\\
&=\sum\limits_{i,j,n,m}(\lbrace \lbrace e^i,e^j \rbrace,e^n\rbrace,e^m)e_me_ne_je_i-2\sum\limits_{j,n}(\mu(e^j,e^n)(e_n)+\phi(\phi(e^j,e^n),e_n))e_j.
\end{align*}
Hence we have
\begin{equation*}
    Q_4(\tilde{J})=\frac{1}{8}\sum\limits_{i,j,n,m}(\lbrace \lbrace e^i,e^j \rbrace,e^n\rbrace,e^m)e_me_ne_je_i-\frac{1}{4}\sum\limits_{j,n}(\mu(e^j,e^n)(e_n)+\phi(\phi(e^j,e^n),e_n))e_j.
\end{equation*}
We have
\begin{align*}
\sum\limits_{j,n}(\mu(e^j,e^n)(e_n)+\phi(\phi(e^j,e^n),e_n))e_j=&Q_2\Big(\sum\limits_{j,n}(\mu(e^j,e^n)(e_n)+\phi(\phi(e^j,e^n),e_n))\wedge e_j\Big)\\
&+\sum\limits_{j,n}\Big((\mu(e^j,e^n)(e_n)+\phi(\phi(e^j,e^n),e_n)),e_j\Big)
\end{align*}
and since
$$\sum\limits_{j,n}(\mu(e^j,e^n)(e_n)+\phi(\phi(e^j,e^n),e_n))\wedge e_j=0$$
we obtain
\begin{equation}\label{formula proof rep of colour Lie type}
    Q_4(\tilde{J})=\frac{1}{8}\sum\limits_{i,j,n,m}(\lbrace \lbrace e^i,e^j \rbrace,e^n\rbrace,e^m)e_me_ne_je_i-\frac{1}{4}\sum\limits_{j,n}\Big(\mu(e^j,e^n)(e_n)+\phi(\phi(e^j,e^n),e_n),e_j\Big).
\end{equation}
On the other hand, we have
\begin{align*}
\pi_*(\Omega(\gg))&=\frac{1}{16}\sum\limits_{k}\sum\limits_{i,j,n,m}(\pi(x^k)(e^j),e^i)(\pi(x_k)(e^m),e^n)e_ie_je_ne_m\\
&=\frac{1}{16}\sum\limits_{k}\sum\limits_{i,j,n,m}B_{\gg}(x^k,\mu(e^j,e^i))B_{\gg}(x_k,\mu(e^m,e^n))e_ie_je_ne_m\\
&=\frac{1}{16}\sum\limits_{i,j,n,m}B_{\gg}\Big(\sum\limits_{k}B_{\gg}(x_k,\mu(e^m,e^n))x^k,\mu(e^j,e^i)\Big)e_ie_je_ne_m\\
&=\frac{1}{16}\sum\limits_{i,j,n,m}B_{\gg}(\mu(e^m,e^n),\mu(e^j,e^i))e_ie_je_ne_m.
\end{align*}
Finally, since
\begin{equation*}
\psi=\frac{1}{12}\sum\limits_{i,j}\phi(e^i,e^j)e_je_i,
\end{equation*}
one can show that
\begin{equation*}
\psi^2=\frac{1}{16}\sum\limits_{i,j,k}(\phi(\phi(e^i,e^j),e^k),e^l)e_le_ke_je_i-\frac{1}{12}\sum\limits_{k,l}(\phi(\phi(e^k,e^l),e_l),e_k)
\end{equation*}
and so
\begin{equation*}
\pi_*(\Omega(\gg))+\psi^2=\frac{1}{16}\sum\limits_{i,j,n,m}(B_{\gg}(\mu(e^m,e^n)(e^j),e^i)+(\phi(\phi(e^m,e^n),e^j),e^i))e_ie_je_ne_m-\frac{1}{12}\sum\limits_{k,l}(\phi(\phi(e^k,e^l),e_l),e_k).
\end{equation*}
Using \eqref{formula proof rep of colour Lie type} we obtain
\begin{equation}\label{formula proof rep of colour Lie type 2}
\pi_*(\Omega(\gg))+\psi^2=\frac{1}{2}Q_4(\tilde{J})+\frac{1}{8}\sum\limits_{j,n}(\mu(e^j,e^n)(e_n),e_j)+\frac{1}{24}\sum\limits_{j,n}(\phi(\phi(e^j,e^n),e_n)),e_j).
\end{equation}
Since $\gt$ is a colour Lie algebra if and only if $J\equiv 0$, it follows from \eqref{formula proof rep of colour Lie type 2} that this is equivalent to
$$\pi_*(\Omega(\gg))+\psi^2\in k.$$
\end{proof}

If there exist a cubic term $\psi$ such that one of the equivalent conditions of this Theorem is satisfied, we say that the representation $\pi : \gg \rightarrow \so_{\epsilon}(V,(\phantom{v} , \phantom{v}))$ is of colour Lie type. If one of this conditions is satisfied for $\psi=0$, we say that the representation is of colour $\mathbb{Z}_2$-Lie type. The scalar involved in this Theorem can be more explicitly characterised.

\begin{cor}\label{cor formule rho(cas)+psi carre}
If $\pi : \gg \rightarrow \so_{\epsilon}(V,(\phantom{v} , \phantom{v}))$ is of colour Lie type for a cubic term $\psi$, then
$$\pi_*(\Omega(\gg))+\psi^2=\frac{1}{24}\Big(Tr_{\epsilon}(ad_{\gt}(\Omega(\gt)))-Tr_{\epsilon}(ad_{\gg}(\Omega(\gg)))\Big).$$
\end{cor}

\begin{proof}
We have
\begin{align*}
    Tr_{\epsilon}(ad_{\gt}(\Omega(\gt)))-Tr_{\epsilon}(ad_{\gg}(\Omega(\gg)))&=\sum\limits_i B_{\gg}(ad(\Omega(\gt))(x^i),x_i)+\sum\limits_i (ad(\Omega(\gt))(e^i),e_i)-\sum\limits_i B_{\gg}(ad(\Omega(\gg))(x^i),x_i)\\
    &=\sum\limits_{i,j}B_{\gg}(\lbrace e^j,\lbrace e_j,x^i\rbrace \rbrace,x_i)+\sum\limits_{i,j}(\lbrace x^j,\lbrace x_j,e^i\rbrace\rbrace,e_i)+\sum\limits_{i,j}(\lbrace e^j,\lbrace e_j,e^i\rbrace\rbrace,e_i)\\
    &=2\sum\limits_{i,j}B_{\gg}(\lbrace e^j,\lbrace e_j,x^i\rbrace \rbrace,x_i)+\sum\limits_{i,j}(\lbrace e^j,\lbrace e_j,e^i\rbrace\rbrace,e_i).
\end{align*}
Since
\begin{align*}
    \sum\limits_{i,j}(\lbrace e^i,\lbrace e_i,x^j\rbrace\rbrace,x_j)&=\sum\limits_{i,j}(\mu(e^j,e^i)(e_i),e_j)
\end{align*}
and
\begin{align*}
    \sum\limits_{i,j}(\lbrace e^j,\lbrace e_j,e^i\rbrace\rbrace,e_i)&=\sum\limits_{i,j}(\mu(e^j,e^i)(e_i),e_j)+\sum\limits_{i,j}(\phi(e^j,\phi(e_j,e^i)),e_i)
\end{align*}
we have
\begin{align*}
    Tr_{\epsilon}(ad_{\gt}(\Omega(\gt)))-Tr_{\epsilon}(ad_{\gg}(\Omega(\gg)))&=3\sum\limits_{i,j}(\mu(e^j,e^i)(e_i),e_j)+\sum\limits_{i,j}(\phi(e^j,\phi(e_j,e^i)),e_i).
\end{align*}
Using \eqref{formula proof rep of colour Lie type 2}, we obtain
$$\pi_*(\Omega(\gg))+\psi^2=\frac{1}{24}\Big(Tr_{\epsilon}(ad_{\gt}(\Omega(\gt)))-Tr_{\epsilon}(ad_{\gg}(\Omega(\gg)))\Big).$$
\end{proof}

Here is a particular family of examples of representations of colour $\mathbb{Z}_2$-Lie type.

\begin{pp} \label{pp ad(Cas) of adjoint representation}
Let $\gg$ be a finite-dimensional colour Lie algebra such that its $\epsilon$-Killing form $B_{\gg}$ is non-degenerate. The $\epsilon$-orthogonal adjoint representation $ad:\gg\rightarrow\so_{\epsilon}(\gg,B_{\gg})$ is of colour $\mathbb{Z}_2$-Lie type and
$$ad_*(\Omega(\gg))=\frac{1}{8}dim_{\epsilon}(\gg).$$
\end{pp}

\begin{proof}
The moment map $\mu$ of $ad:\gg\rightarrow\so_{\epsilon}(\gg,B_{\gg})$ satisfies to $\mu(u,v)=\lbrace u,v\rbrace$ for $u,v\in\gg$ and then $\gt=\gg\oplus\gg$ is a $\mathbb{Z}_2$-graded colour Lie algebra. From Theorem \ref{thm N(mu+phi)=0} and Corollary \ref{cor formule rho(cas)+psi carre} we have
$$ad_*(\Omega(\gg))=\frac{1}{24}\Big(Tr_{\epsilon}(ad_{\gt}(\Omega(\gt)))-Tr_{\epsilon}(ad_{\gg}(\Omega(\gg)))\Big).$$
The $\epsilon$-quadratic form $B_{\gg}$ is the $\epsilon$-Killing form of $\gg$ and then the $\epsilon$-quadratic form $B_{\gt}=B_{\gg}\perp B_{\gg}$ on $\gt$ satisfies to
$$B_{\gt}=\frac{1}{2}K_{\gt}$$
where $K_{\gt}$ is the $\epsilon$-Killing form of $\gt$. Hence, if $\Omega_K(\gt)\in U_{\epsilon}(\gt)$ is the Casimir element corresponding to $K_{\gt}$, we have 
$$\Omega(\gt)=2\Omega_K(\gt).$$
Thus,
$$Tr_{\epsilon}(ad_{\gt}(\Omega(\gt)))=2Tr_{\epsilon}(ad_{\gt}(\Omega_K(\gt)))=2dim_{\epsilon}(\gt)=4dim_{\epsilon}(\gg)$$
and so
$$ad_*(\Omega(\gg))=\frac{1}{8}dim_{\epsilon}(\gg).$$
\end{proof}

Comparing the scalar coming from the previous example with the action of the Casimir element on an oscillator representation we obtain a strange Freudenthal and de Vries formula for basic colour Lie algebras.

\begin{thm}\label{thm FdV formula}
Suppose that $char(k)=0$. Let $\gg$ be a finite-dimensional basic colour Lie algebra in the sense of Subsection \ref{subsection split CLA} and such that its $\epsilon$-Killing form $B_{\gg}$ is non-degenerate. We have
$$dim_{\epsilon}(\gg)=24B_{\gg^*}(\rho,\rho),$$
where $\rho$ is the Weyl vector of $\gg$ with respect to a Cartan subalgebra $\hh$ of $\gg$ and a choice of a full set of positive roots for the action of $ad(\hh)$.
\end{thm}

\begin{proof}
We use here the same notation as in Subsection \ref{subsection split CLA} and we can suppose without loss of generality that $k$ is algebraically closed.
\vspace{0.2cm}

From Proposition \ref{pp inner product between roots} we know that $\gg^+$ and $\gg^-$ are isotropic for $B_{\gg}$. Let $\tilde{W},\tilde{W}^*$ be maximal isotropic subspaces of $\hh$ and let $U$ be $\lbrace 0 \rbrace$ or a one-dimensional vector space of degree $0$ orthogonal to $W$ and $W'$ such that $\hh=\tilde{W}^*\oplus\tilde{W}\oplus U$. If we define $W=\tilde{W}\oplus \gg^-$ and $W^*=\tilde{W}^*\oplus\gg^+$ then $W$ and $W^*$ are maximal isotropic subspaces of $\gg$ and we can consider the oscillator representation $m:C_{\epsilon}(\gg,B_{\gg})\rightarrow {\rm End}(\Lambda_{\epsilon}(W))$ (see Definition \ref{def oscillator}). Composing $m$ with $ad_*:\gg\rightarrow C_{\epsilon}(V,B_{\gg})$ we obtain a representation $m_*: \gg\rightarrow {\rm End}(\Lambda_{\epsilon}(W))$.
\vspace{0.2cm}

For $h\in\hh$ we have that
\begin{align*}
    ad_*(h)&=\frac{1}{4}\sum\limits_{\alpha\in\Delta^+}(\lbrace h,e_{-\alpha}\rbrace e_\alpha+\epsilon(\alpha)\lbrace h,e_\alpha\rbrace e_{-\alpha})\\
    &=\frac{1}{4}\sum\limits_{\alpha\in\Delta^+}(-\alpha(h)e_{-\alpha}e_\alpha+\epsilon(\alpha)\alpha(h)e_\alpha e_{-\alpha})\\
    &=\frac{1}{2}\sum\limits_{\alpha\in\Delta^+}(-\alpha(h)e_{-\alpha}e_\alpha+\epsilon(\alpha)\alpha(h))
\end{align*}
and since $e_\alpha\in W^*$ we have
\begin{equation} \label{proof m(h)=rho(h)}
    m_*(h)(1)=\rho(h)1.
\end{equation}
Let $\gamma \in \Delta^+$. Let $\lbrace h_i \rbrace$ be a basis of $\hh$ and let $\lbrace h^i\rbrace$ be its dual basis. We have
\begin{align*}
    ad_*(e_\gamma)&=\frac{1}{4}\Big(\sum\limits_i\lbrace e_\gamma,h^i\rbrace h_i+\sum\limits_{\alpha\in\Delta^+}( \lbrace e_\gamma,e_{-\alpha}\rbrace e_\alpha+\epsilon(\alpha)\lbrace e_\gamma,e_{\alpha}\rbrace e_{-\alpha})\Big)
\end{align*}
and since $e_\alpha\in W^*$ we have
\begin{equation*}
        m_*(e_\gamma)(1)=\frac{1}{4}\Big(\sum\limits_i m(\lbrace e_\gamma,h^i\rbrace h_i)(1)+\sum\limits_{\alpha\in\Delta^+}\epsilon(\alpha)m(\lbrace e_\gamma,e_{\alpha}\rbrace e_{-\alpha})(1)\Big).
\end{equation*}
Using Proposition \ref{pp roots spaces orthogonaux}, we have
$$\sum\limits_{\alpha\in\Delta^+}\epsilon(\alpha)m(\lbrace e_\gamma,e_{\alpha}\rbrace e_{-\alpha})(1)=\frac{1}{2}\sum\limits_{\alpha\in \Delta^+}\epsilon(\alpha)B_{\gg}(\lbrace e_\gamma,e_\alpha\rbrace,e_{-\alpha})=0$$
and so
$$m_*(e_\gamma)(1)=\frac{1}{4}\sum\limits_i m(\lbrace e_\gamma,h^i\rbrace h_i)(1).$$
Let $\lbrace H_i \rbrace$ be a basis of $\tilde{W}^*$ and let $\lbrace H^i\rbrace$ be the basis of $\tilde{W}$ such that $B_{\gg}(H_i,H^j)=\delta_{ij}$. If $U\neq \lbrace 0 \rbrace$ let $u$ be a basis of $U$ such that $B_{\gg}(u,u)=1$. The basis $\lbrace h_i\rbrace$ can be chosen such that $\lbrace h_i\rbrace=\lbrace H_i\rbrace \cup \lbrace H^i \rbrace$ if $U=\lbrace 0 \rbrace$ or such that $\lbrace h_i\rbrace=\lbrace H_i\rbrace \cup \lbrace H^i \rbrace\cup \lbrace u\rbrace$ if $U\neq\lbrace 0 \rbrace$. Since $H_i\in W^*$ we have
\begin{align*}
    \frac{1}{4}\sum\limits_i m(\lbrace e_\gamma,H_i\rbrace H^i+\lbrace e_\gamma,H^i\rbrace H_i)(1)&=\frac{1}{4}\sum\limits_i m(\lbrace e_\gamma,H_i\rbrace H^i)(1)\\
    &=-\frac{1}{4}\sum\limits_i \gamma(H_i)m(e_\gamma H^i)(1)\\
    &=-\frac{1}{2}\sum\limits_i \gamma(H_i)B_{\gg}(e_\gamma,H^i)\\
    &=0.
\end{align*}
and if $U\neq \lbrace 0 \rbrace$ we also have
\begin{equation*}
\frac{1}{4}m(\lbrace e_\gamma,u\rbrace u)(1)=\frac{-\gamma(u)}{4}m(e_\gamma u)(1)=\frac{-\gamma(u)}{4}m(e_\gamma)(1)=0
\end{equation*}
and so
\begin{equation}\label{proof m(e_gamma)=0}
m_*(e_\gamma)(1)=0.
\end{equation}
Hence, from Equations \eqref{proof m(h)=rho(h)} and \eqref{proof m(e_gamma)=0}, we have that $m_*(\gg)(1)$ is a representation of $\gg$ of highest weight $\rho$. By Proposition \ref{pp action Casimir on highest weight} we have
$$m_*(\Omega(\gg))|_{m_*(\gg)(1)}=3B_{\gg^*}(\rho,\rho).$$
Since $m_*=m\circ ad_*$, by Proposition \ref{pp ad(Cas) of adjoint representation}, we also have
$$m_*(\Omega(\gg))|_{m_*(\gg)(1)}=\frac{1}{8}dim_\epsilon(\gg)$$
and so
$$dim_{\epsilon}(\gg)=24B_{\gg^*}(\rho,\rho).$$
\end{proof}

\section{The cubic Dirac operator}\label{section cubic Dirac operator}

In this section, we define cubic Dirac operators for $\epsilon$-orthogonal representations of $\epsilon$-quadratic colour Lie algebras and we prove a Parthasarathy formula (\cite{Parthasarathy72}).
\vspace{0.2cm}

Let $\pi:\gg\rightarrow\so_{\epsilon}(V,(\phantom{v},\phantom{v}))$ be a finite-dimensional $\epsilon$-orthogonal representation of colour Lie type of a finite-dimensional $\epsilon$-quadratic colour Lie algebra $(\gg,B_{\gg})$ and consider the associated $\epsilon$-quadratic colour Lie algebra $(\gt,B_{\gt})$ where $\gt=\gg\oplus V$ and $B_{\gt}=B_{\gg}\perp (\phantom{v},\phantom{v})$. The $\epsilon$-enveloping algebra $U_{\epsilon}(\gt)$ is $\Gamma$-graded, the $\epsilon$-Clifford algebra $C_{\epsilon}(V,(\phantom{v},\phantom{v}))$ is $\mathbb{Z}_2\times \Gamma$-graded, hence $U_{\epsilon}(\gt)\otimes C_{\epsilon}(V,(\phantom{v},\phantom{v}))$ is $\mathbb{Z}_2\times\Gamma$-graded and we consider the product
$$a\otimes b\cdot c\otimes d:=\tilde{\epsilon}(b,c)ac\otimes bd \qquad \forall a\otimes b,c\otimes d\in U_{\epsilon}(\gt)\otimes C_{\epsilon}(V,(\phantom{v},\phantom{v})).$$
Then $U_{\epsilon}(\gt)\otimes C_{\epsilon}(V,(\phantom{v},\phantom{v}))$ is a $(\mathbb{Z}_2\times\Gamma,\tilde{\epsilon})$-colour Lie algebra for the bracket
$$\lbrace x,y \rbrace:=x\cdot y-\tilde{\epsilon}(x,y)y\cdot x \qquad \forall x,y \in U_{\epsilon}(\gt)\otimes C_{\epsilon}(V,(\phantom{v},\phantom{v})).$$
We now define the cubic Dirac operator of the representation $\pi:\gg\rightarrow\so_{\epsilon}(V,(\phantom{v},\phantom{v}))$ of colour Lie type.

\begin{df}
Let $\lbrace e_i\rbrace$ be a basis of $V$ and let $\lbrace e^i\rbrace$ be its dual basis. Define $D \in U_{\epsilon}(\gt)\otimes C_{\epsilon}(V,(\phantom{v},\phantom{v}))$ by
$$D=\sum\limits_i e^i\otimes e_i,$$
and the cubic Dirac operator $\tilde{D}\in U_{\epsilon}(\gt)\otimes C_{\epsilon}(V,(\phantom{v},\phantom{v}))$ by
$$\tilde{D}=D+1\otimes \psi,$$
where $\psi$ is the cubic term of the colour Lie type representation $\pi:\gg\rightarrow\so_{\epsilon}(V,(\phantom{v},\phantom{v}))$.
\end{df}

Define the diagonal embedding $\Delta : \gg\rightarrow U_{\epsilon}(\gt)\otimes C_{\epsilon}(V,(\phantom{v},\phantom{v}))$ by
$$\Delta(x)=x\otimes 1 + 1\otimes \pi_*(x)\qquad \forall x\in \gg.$$
This is a colour Lie algebra morphism and we extend it to $\Delta : U_{\epsilon}(\gg)\rightarrow U_{\epsilon}(\gt)\otimes C_{\epsilon}(V,(\phantom{v},\phantom{v}))$.
\vspace{0.2cm}

We calculate the square of the cubic Dirac operator of a representation of colour Lie type analogously to the Parthasarathy formula for the square of the Dirac operator of a Cartan decomposition of a reductive Lie algebra.

\begin{thm} \label{thm square of Dirac}
Let $\pi:\gg\rightarrow\so_{\epsilon}(V,(\phantom{v},\phantom{v}))$ be a finite-dimensional $\epsilon$-orthogonal representation of colour Lie type of a finite-dimensional $\epsilon$-quadratic colour Lie algebra $(\gg,B_{\gg})$. Let $\gt$ be the associated colour Lie algebra and let $\tilde{D}$ be the cubic Dirac operator. We have
$$\tilde{D}^2=\Omega(\gt)\otimes 1-\Delta(\Omega(\gg))+\frac{1}{24}\Big(Tr_{\epsilon}(ad_{\gt}(\Omega(\gt)))-Tr_{\epsilon}(ad_{\gg}(\Omega(\gg)))\Big) 1\otimes 1,$$
where $\Omega(\gt)$ (resp. $\Omega(\gg)$) is the Casimir element of $\gt$ (resp. $\gg$).
\end{thm}

\begin{proof}
Let $\psi$ be the cubic term of the representation of colour Lie type $\pi : \gg \rightarrow \so_{\epsilon}(V,(\phantom{v} , \phantom{v}))$. We have
\begin{align*}
D^2&=\frac{1}{2}\sum\limits_{i,j}\Big(e^i\otimes e_i \cdot e^j\otimes e_j+e^j\otimes e_j \cdot e^i\otimes e_i \Big)\\
&=\frac{1}{2}\sum\limits_{i,j}\Big(\epsilon(e_j,e_i) e^ie^j\otimes e_ie_j+\epsilon(e_i,e_j) e^je^i\otimes e_je_i \Big)\\
&=\frac{1}{2}\sum\limits_{i,j}\Big(\epsilon(e_j,e_i) e^ie^j\otimes e_ie_j-e^je^i\otimes e_ie_j \Big)+\sum\limits_{i,j}\epsilon(e_i,e_j)e^je^i\otimes(e_j,e_i)\\
&=-\frac{1}{2}\sum\limits_{i,j}\lbrace e^j,e^i \rbrace \otimes e_ie_j+\sum\limits_{i,j}\epsilon(e_i,e_j)e^je^i\otimes(e_j,e_i)\\
&=-\frac{1}{2}\sum\limits_{i,j}\mu(e^j,e^i)\otimes e_ie_j-\frac{1}{2}\sum\limits_{i,j}\phi(e^j,e^i)\otimes e_ie_j+\sum\limits_{i,j}\epsilon(e_i,e_j)e^je^i\otimes (e_j,e_i).
\end{align*}
Firstly, using Equation \eqref{eq decomp (v,ei)ei}, we have
\begin{align*}
\sum\limits_{i,j}\epsilon(e_i,e_j)e^je^i\otimes (e_j,e_i)&=\sum\limits_{i,j} e^j(e_i,e_j)e^i\otimes 1\\
&=\Big(\Omega(\gt)-\Omega(\gg)\Big)\otimes 1,
\end{align*}
using Equations \eqref{eq decomp (v,ei)ei} and \eqref{eq pi in clifford}, we also have
\begin{align*}
-\frac{1}{2}\sum\limits_{i,j}\mu(e^j,e^i)\otimes e_ie_j&=-\frac{1}{2}\sum\limits_{i,j}\sum\limits_k B_{\gg}(x_k,\mu(e^j,e^i))x^k\otimes e_ie_j\\
&=-\frac{1}{2}\sum\limits_{i,j}\sum\limits_k (\pi(x_k)(e^j),e^i)x^k\otimes e_ie_j\\
&=-\frac{1}{2}\sum\limits_k x^k\otimes\sum\limits_j \pi(x_k)(e^j)e_j\\
&=-2\sum\limits_k x^k\otimes \pi_*(x_k)
\end{align*}
and hence
\begin{equation}\label{eq D^2 1}
D^2=-2\sum\limits_k x^k\otimes \pi_*(x_k)-\frac{1}{2}\sum\limits_{i,j}\phi(e^j,e^i)\otimes e_ie_j+\Big(\Omega(\gt)-\Omega(\gg)\Big)\otimes 1.
\end{equation}
We have
$$1\otimes \psi\cdot D=\frac{1}{12}\sum\limits_{i,j,k,l}\epsilon(e_l)e_l\otimes \psi(e^k,e^j,e^i)e_ie_je_ke^l.$$
Since
$$e_ie_je_ke^l=-\epsilon(e_l,e_i+e_j+e_k)e^le_ie_je_k+2\epsilon(e_l,e_j+e_k)\delta_{il}e_je_k-2\epsilon(e_l,e_k)\delta_{jl}e_ie_k+2\delta_{kl}e_ie_j$$
and
\begin{align*}
&\frac{1}{12}\sum\limits_{i,j,k,l}\epsilon(e_l)e_l\otimes \psi(e^k,e^j,e^i) \Big(2\epsilon(e_l,e_j+e_k)\delta_{il}e_je_k-2\epsilon(e_l,e_k)\delta_{jl}e_ie_k+2\delta_{kl}e_ie_j\Big)=\frac{1}{2}\sum_{j,k}\phi(e^k,e^j)\otimes e_je_k
\end{align*}
we obtain
\begin{equation*}
1\otimes \psi\cdot D=-\frac{1}{12}\sum\limits_{i,j,k,l}\epsilon(e_l)e_l\otimes \psi(e^k,e^j,e^i)e^le_ie_je_k+\frac{1}{2}\sum_{j,k}\phi(e^k,e^j)\otimes e_je_k.
\end{equation*}
Since
$$D\cdot 1\otimes \psi=\frac{1}{12}\sum\limits_{i,j,k,l}\epsilon(e_l)e_l\otimes \psi(e^k,e^j,e^i)e^le_ie_je_k$$
we have
\begin{align*}
D\cdot 1\otimes \psi + 1\otimes \psi\cdot D&=\sum\limits_le^l\otimes e_l\cdot 1\otimes \frac{1}{12}\sum\limits_{i,j,k}\psi(e^k,e^j,e^i)e_ie_je_k+1\otimes \frac{1}{12}\sum\limits_{i,j,k}\psi(e^k,e^j,e^i)e_ie_je_k\cdot\sum\limits_le^l\otimes e_l\\
&=\frac{1}{2}\sum\limits_{i,j}\phi(e^j,e^i)\otimes e_ie_j
\end{align*}
and so from Equation \eqref{eq D^2 1} we obtain
\begin{equation*}
D^2=-2\sum\limits_k x^k\otimes \pi_*(x_k)+\Big(\Omega(\gt)-\Omega(\gg)\Big)\otimes 1+ 1\otimes \psi^2.
\end{equation*}
We have
\begin{align*}
\Delta(\Omega(\gg))&=\sum\limits_k\Delta(x^k)\Delta(x_k)\\
&=\sum\limits_k \Big(x^k\otimes 1 + 1\otimes \pi_*(x^k)\Big)\Big(x_k\otimes 1 + 1\otimes \pi_*(x_k)\Big)\\
&=\sum\limits_kx^kx_k\otimes1+\sum\limits_kx^k\otimes \pi_*(x_k)+\sum\limits_k\epsilon(x_k)x_k\otimes \pi_*(x^k)+\sum_k1\otimes\pi_*(x^k)\pi_*(x_k)\\
&=\Omega(\gg)\otimes 1+2\sum_k x^k\otimes\pi_*(x_k)+1\otimes\pi_*(\Omega(\gg)),
\end{align*}
hence
$$(D+1\otimes\psi)^2=\Omega(\gt)\otimes 1-\Delta(\Omega(\gg))+\Big(1\otimes\pi_*(\Omega(\gg))+1\otimes \psi^2\Big)$$
and from Corollary \ref{cor formule rho(cas)+psi carre} we obtain
$$\tilde{D}^2=\Omega(\gt)\otimes 1-\Delta(\Omega(\gg))+\frac{1}{24}\Big(Tr_{\epsilon}(ad_{\gt}(\Omega(\gt)))-Tr_{\epsilon}(ad_{\gg}(\Omega(\gg)))\Big) 1\otimes 1.$$
\end{proof}

\begin{cor}
Suppose that $char(k)=0$, the colour Lie algebras $\gt$ and $\gg$ are basic in the sense of Subsection \ref{subsection split CLA} and $B_{\gt}$ (resp. $B_{\gg}$) is a multiple of the $\epsilon$-Killing form of $\gt$ (resp. $\gg$). We have
$$\tilde{D}^2=\Omega(\gt)\otimes 1-\Delta(\Omega(\gg))+\Big(B_{\gt^*}(\rho_{\gt},\rho_{\gt})-B_{\gg^*}(\rho_{\gg},\rho_{\gg})\Big) 1\otimes 1,$$
where $\rho_{\gt}$ (resp. $\rho_{\gg}$) is the Weyl vector of $\gt$ (resp. $\gg$) with respect to a Cartan subalgebra $\hh_{\gt}$ (resp. $\hh_{\gg}$) of $\gt$ (resp. $\gg$) and a choice of a full set of positive roots for the action of $ad(\hh_{\gt})$ (resp. $ad(\hh_{\gg})$).
\end{cor}

\begin{proof}
Let $c\in k^*$ be such that $B_{\gg}=cK_{\gg}$, where $K_{\gg}$ is the $\epsilon$-Killing form of $\gg$. Hence, if $\Omega_K(\gg)\in U_{\epsilon}(\gg)$ is the Casimir element corresponding to $K_{\gg}$, we have 
$$\Omega(\gg)=\frac{1}{c}\Omega_K(\gg).$$
Thus,
$$Tr_{\epsilon}(ad_{\gg}(\Omega(\gg)))=\frac{1}{c}Tr_{\epsilon}(ad_{\gg}(\Omega_K(\gg)))=\frac{1}{c}dim_{\epsilon}(\gg)$$
and by the strange Freudenthal-de Vries formula (Theorem \ref{thm FdV formula}) we obtain
$$Tr_{\epsilon}(ad_{\gg}(\Omega(\gg)))=\frac{24}{c}K_{\gg^*}(\rho_{\gg},\rho_{\gg}).$$
Since $B_{\gg^*}=\frac{1}{c}K_{\gg^*}$, we have
$$Tr_{\epsilon}(ad_{\gg}(\Omega(\gg)))=24B_{\gg^*}(\rho_{\gg},\rho_{\gg}).$$
\end{proof}

\section{Differential complex induced by the cubic Dirac operator}\label{section cohomology}

In this section we show that the cubic Dirac operator of a representation of colour Lie type induces a differential complex and we calculate the associated cohomology, which is the analogue of the algebraic Vogan conjecture proved by Huang and Pand\v{z}i\'c (\cite{HuangPandzic02}) for colour Lie algebras.

\begin{df}
Let $\pi:\gg\rightarrow\so_{\epsilon}(V,(\phantom{v},\phantom{v}))$ be a finite-dimensional $\epsilon$-orthogonal representation of colour Lie type of a finite-dimensional $\epsilon$-quadratic colour Lie algebra $(\gg,B_{\gg})$. Let $\gt$ be the associated colour Lie algebra and let $\tilde{D}$ be the cubic Dirac operator. Define the induced differential $d \in {\rm End}\Big(U_{\epsilon}(\gt)\otimes C_{\epsilon}(V,(\phantom{v},\phantom{v}))\Big)$ by
$$d(a):=\lbrace \tilde{D},a\rbrace \qquad \forall a \in U_{\epsilon}(\gt)\otimes C_{\epsilon}(V,(\phantom{v},\phantom{v})).$$
\end{df}

Consider the $\gg$-invariant subspace of $U_{\epsilon}(\gt)\otimes C_{\epsilon}(V,(\phantom{v},\phantom{v}))$ defined by
$$\Big(U_{\epsilon}(\gt)\otimes C_{\epsilon}(V,(\phantom{v},\phantom{v}))\Big)^{\gg}=\lbrace  a\in U_{\epsilon}(\gt)\otimes C_{\epsilon}(V,(\phantom{v},\phantom{v})) ~ | ~ \lbrace \Delta(x), a \rbrace=0 \quad \forall x \in \gg \rbrace,$$
and define $\mathcal{Z}_{\epsilon}(\Delta(\gg))\subset U_{\epsilon}(\gt)\otimes C_{\epsilon}(V,(\phantom{v},\phantom{v}))$ by $\mathcal{Z}_{\epsilon}(\Delta(\gg)):=\Delta(\mathcal{Z}_{\epsilon}(\gg))$.

\begin{pp}\label{pp d^2=0} In $\Big(U_{\epsilon}(\gt)\otimes C_{\epsilon}(V,(\phantom{v},\phantom{v}))\Big)^{\gg}$, we have $d^2=0$ and we have $\mathcal{Z}_{\epsilon}(\Delta(\gg))\subset Ker(d)$.
\end{pp}

\begin{proof}
We first prove that $d^2=0$. Let $a\in U_{\epsilon}(\gt)\otimes C_{\epsilon}(V,(\phantom{v},\phantom{v}))$. We have
\begin{align*}
d^2(a)&=d(\tilde{D}a-\tilde{\epsilon}(\tilde{D},a)a\tilde{D})\\
&=\tilde{D}^2a-\tilde{\epsilon}(\tilde{D},\tilde{D}a)\tilde{D}a\tilde{D}-\tilde{\epsilon}(\tilde{D},a)\tilde{D}a\tilde{D}+\tilde{\epsilon}(\tilde{D},a+a\tilde{D})a\tilde{D}^2
\end{align*}
and since $\tilde{\epsilon}(\tilde{D},a)=(-1)^{|a|}$ and $\tilde{\epsilon}(\tilde{D},\tilde{D}a)=-(-1)^{|a|}$ we obtain 
$$d^2(a)=\tilde{D}^2a-a\tilde{D}^2.$$
From Theorem \ref{thm square of Dirac}, we have
$\tilde{D}^2=\Omega(\gt)\otimes 1-\Delta(\Omega(\gg))+c 1\otimes 1$ where $c\in k$ and hence
$$d^2(a)=\Omega(\gt)\otimes 1\cdot a-a\cdot \Omega(\gt)\otimes1-\Delta(\Omega(\gg))\cdot a+a\cdot \Delta(\Omega(\gg)).$$
Since $\Omega(\gt)\in\mathcal{Z}_{\epsilon}(\gt)$, we obtain
$$d^2(a)=-\Delta(\Omega(\gg))\cdot a+a\cdot \Delta(\Omega(\gg)).$$
And hence, if $a\in \Big(U_{\epsilon}(\gt)\otimes C_{\epsilon}(V,(\phantom{v},\phantom{v}))\Big)^{\gg}$, then we obtain that $d^2(a)=0.$
\vspace{0.2cm}

The assertion $\mathcal{Z}_{\epsilon}(\Delta(\gg))\subset Ker(d)$ follows from a straightforward calculation since $D$ corresponds to $Id_V$ and since the cubic term $\psi$ is invariant by the action of $\gg$.
\end{proof}

We now prove the analogue of the Vogan conjecture for colour Lie algebras.

\begin{thm} \label{thm kernel of d}
Suppose $char(k)=0$. Let $\pi:\gg\rightarrow\so_{\epsilon}(V,(\phantom{v},\phantom{v}))$ be a finite-dimensional $\epsilon$-orthogonal representation of colour Lie type of a finite-dimensional $\epsilon$-quadratic colour Lie algebra $(\gg,B_{\gg})$. Let $\gt$ be the associated colour Lie algebra, let $\tilde{D}$ be the cubic Dirac operator and let $d$ be its induced differential on $(U_{\epsilon}(\gt)\otimes C_{\epsilon}(V,(\phantom{v},\phantom{v})))^{\gg}$. Then we have $$Ker(d)=\mathcal{Z}_{\epsilon}(\Delta(\gg))\oplus Im(d).$$
\end{thm}

\begin{proof}
We first need the following lemma.

\begin{lm}\label{lm koszul complex}
Suppose that $\epsilon(v)=1$ for all $v\in V$ or suppose $char(k)=0$. Let $V$ be a finite-dimensional $\Gamma$-graded vector space. Let $X_n(V):=S_{\epsilon}(V)\otimes \Lambda^n_{\epsilon}(V)$ and $d_{n+1} : X_{n+1}(V)\rightarrow X_n(V)$ be defined by
$$d_{n+1}(x\otimes y_0\wedge \cdots \wedge y_n)=\sum \limits_{i=0}^n (-1)^{i}\epsilon(y_0+\cdots+y_{i-1},y_i) xy_i\otimes y_0\wedge\cdots \wedge \hat{y_i} \wedge \cdots \wedge y_n.$$
Then
\begin{equation} \label{complex chain S otimes Lambda}
\ldots \xrightarrow{d_{n+1}} X_n(V) \xrightarrow{d_n} \ldots \xrightarrow{d_1} X_0(V) \xrightarrow{d_0} k \rightarrow \lbrace 0\rbrace,
\end{equation}
where $d_0 : S_{\epsilon}(V)\otimes k \rightarrow k$ is the augmentation map, is an exact chain complex. That is to say that $(X_{\bullet}(V),d)$ is a standard resolution of the trivial $S_{\epsilon}(V)$-module $k$.
\end{lm}
\begin{proof}
One can check that $d_n\circ d_{n+1}=0$ and that \eqref{complex chain S otimes Lambda} is exact in $X_{-1}(V)=k$. We now show by induction on $dim(V)$ that this chain complex is exact.\\

\noindent
$\bullet$ Initialisation: Suppose that $dim(V)=1$ and let $v$ be a basis of $V$. If $\epsilon(v)=1$, we have $\Lambda^n_{\epsilon}(V)=\lbrace 0 \rbrace$ for all $n\geq 2$ and then the chain complex $(X_{\bullet}(V),d)$ is of the form
\begin{equation} \label{complex chain dim(v)=1}
\lbrace 0 \rbrace \rightarrow S(V)\otimes V \xrightarrow{d_1} S(V)\otimes k \xrightarrow{d_0} k \rightarrow \lbrace 0\rbrace.
\end{equation}
The map $d_1$ is clearly injective. We have $Ker(d_0)=\sum\limits_{i\geq 1}S^i_{\epsilon}(V)\otimes k$ and since for $x\otimes v \in S_{\epsilon}(V)\otimes V$ we have $d_1(x\otimes v)=xv\otimes 1$ we obtain that $Ker(d_0)=Im(d_1)$. Hence \eqref{complex chain dim(v)=1} is a short exact sequence.
\vspace{0.1cm}

If $\epsilon(v)=-1$, every $y$ in $X_n(V)=\Lambda(V)\otimes S^n(V)$ is of the form $y=(\sum\limits_i \alpha_i+\beta_i v)\otimes v^n$ and then, since $char(k)=0$, $d_n(y)=\sum\limits_i n\alpha_iv\otimes v^{n-1}$ is equal to $0$ if and only if $y=\sum\limits_i \beta_i v\otimes v^n$ that is to say if and only if $y\in Im(d_{n+1})$ and so the chain complex \eqref{complex chain S otimes Lambda} is exact.
\vspace{0.3cm}

\noindent
$\bullet$ Induction: Let $n\in \mathbb{N}$ and suppose that the chain complex \eqref{complex chain S otimes Lambda} is exact for every $\Gamma$-graded vector space $V$ of dimension $n$. Let $V$ be a $\Gamma$-graded vector space of dimension $n+1$. Let $v$ be an homogeneous element of $V$ and let $W\subset V$ be a $\Gamma$-graded vector space such that $V=W\oplus kv$.
\vspace{0.2cm}

Since $\Lambda_{\epsilon}(V)\cong\Lambda_{\epsilon}(W) \otimes \Lambda_{\epsilon}(kv)$ and since the family $\lbrace v^{\wedge i}, ~ \forall i \in \mathbb{N} \rbrace$ generates $\Lambda_{\epsilon}(kv)$, we define a filtration of $X_n(V)$ by
$$F_p(X_n(V)):=S_{\epsilon}(V)\otimes \overset{p}{\underset{i=0}{\oplus}}\Lambda_{\epsilon}^{n-i}(W)\wedge v^{\wedge i}.$$
The multiplication on the right by $v^{\wedge p}$ defines an isomorphism between the chain complex $(gr(X)_{\bullet},gr(d))$ and the chain complex $(S_{\epsilon}(V)\underset{S_{\epsilon}(W)}{\otimes} X_{\bullet-p}(W),1\otimes d)$. The chain complex $(X_{\bullet}(W),d)$ is exact by assumption. Moreover, $S_{\epsilon}(V)$ is a free $S_{\epsilon}(W)$-module, so the chain complex $(S_{\epsilon}(V)\underset{S_{\epsilon}(W)}{\otimes} X_{\bullet}(W),1\otimes d)$ is exact and hence $(X_{\bullet}(V),d)$ is exact.
\end{proof}
\vspace{0.4cm}

Denote by $\alpha(x)$ the $\mathbb{Z}_2$-degree of $x\in U_{\epsilon}(\gt)\otimes C_{\epsilon}(V,(\phantom{v},\phantom{v}))$. The standard filtration of $U_{\epsilon}(\gt)$ induces a filtration of $U_{\epsilon}(\gt)\otimes C_{\epsilon}(V,(\phantom{v},\phantom{v}))$ and since
$$\gg\cdot F_n(U_{\epsilon}(\gt)\otimes C_{\epsilon}(V,(\phantom{v},\phantom{v})))\subseteq F_n(U_{\epsilon}(\gt)\otimes C_{\epsilon}(V,(\phantom{v},\phantom{v}))),$$
then we have a filtration of $(U_{\epsilon}(\gt)\otimes C_{\epsilon}(V,(\phantom{v},\phantom{v})))^{\gg}$ by 
$$F_n((U_{\epsilon}(\gt)\otimes C_{\epsilon}(V,(\phantom{v},\phantom{v})))^{\gg}):=F_n(U_{\epsilon}(\gt)\otimes C_{\epsilon}(V,(\phantom{v},\phantom{v})))^{\gg}.$$
Since $gr(U_{\epsilon}(\gt))\cong S_{\epsilon}(\gt)$ (see \cite{Scheunert79}) and $S_{\epsilon}(\gg\oplus V)\cong S_{\epsilon}(\gg)\otimes S_{\epsilon}(V)$ the map
$$d : U_{\epsilon}(\gt)\otimes C_{\epsilon}(V,(\phantom{v},\phantom{v}))\rightarrow U_{\epsilon}(\gt)\otimes C_{\epsilon}(V,(\phantom{v},\phantom{v}))$$ induces
$$gr(d) : S_{\epsilon}(\gg)\otimes S_{\epsilon}(V)\otimes C_{\epsilon}(V,(\phantom{v},\phantom{v}))\rightarrow S_{\epsilon}(\gg)\otimes S_{\epsilon}(V)\otimes C_{\epsilon}(V,(\phantom{v},\phantom{v})).$$
For $a\otimes x\otimes y \in S_{\epsilon}(\gg)\otimes S_{\epsilon}(V)\otimes C_{\epsilon}(V,(\phantom{v},\phantom{v}))$ we have
\begin{align*}
gr(d)(a\otimes x \otimes y)&=a\otimes\sum\limits_i  x e^i\otimes \Big(e_iy-(-1)^{\alpha(y)}\epsilon(e_i,y)ye_i\Big),
\end{align*}
where $\lbrace e_i\rbrace$ is a basis of $V$ and $\lbrace e^i\rbrace$ its dual basis. The quantisation map $Q:\Lambda_{\epsilon}(V)\rightarrow C_{\epsilon}(V,(\phantom{v},\phantom{v}))$ is a canonical isomorphism of $\Gamma$-graded vector spaces. If we consider $d':S_{\epsilon}(\gg)\otimes S_{\epsilon}(V)\otimes \Lambda_{\epsilon}(V)\rightarrow S_{\epsilon}(\gg)\otimes S_{\epsilon}(V)\otimes \Lambda_{\epsilon}(V)$ given by 
$$d'(a\otimes x\otimes y_0\wedge \cdots \wedge y_n)=2a\otimes\sum \limits_{i=0}^n (-1)^{i}\epsilon(y_0+\cdots+y_{i-1},y_i) xy_i\otimes y_0\wedge\cdots \wedge \hat{y_i} \wedge \cdots \wedge y_n,$$
using
$$Q(y_0\wedge \ldots \wedge y_n)=\frac{1}{n+1}\sum\limits_{\sigma \in S(\llbracket 1 ,n\rrbracket,\lbrace n+1 \rbrace)}p(\sigma,y_0+\ldots+y_n)Q(y_{\sigma(0)}\wedge\ldots\wedge y_{\sigma(n-1)})y_{\sigma(n)},$$
where $S(\llbracket 1 ,n\rrbracket,\lbrace n+1 \rbrace)$ denotes the shuffle permutations of $\llbracket 1 ,n\rrbracket$ and $\lbrace n+1 \rbrace$ in $\llbracket 1 ,n+1\rrbracket$, then one can show by induction on $n$ that $gr(d)\circ Q=Q\circ d'$.\\

Hence, by Lemma \ref{lm koszul complex}, the kernel $Ker(d')\subseteq S_{\epsilon}(V)\otimes \Lambda_{\epsilon}(V)$ satisfies
$$Ker(d')=Im(d)\oplus k\cdot 1\otimes 1,$$
and so the kernel $Ker(d')\subseteq (S_{\epsilon}(\gg)\otimes S_{\epsilon}(V)\otimes \Lambda_{\epsilon}(V))^{\gg}$ satisfies
\begin{equation}\label{proof thm kernel of d 1}
Ker(d')=Im(d)\oplus  (S_{\epsilon}(\gg))^{\gg}\otimes 1\otimes 1.
\end{equation}
Now we go back to the study of the kernel of $d$. We have
$$\mathcal{Z}_{\epsilon}(\Delta(\gg))\oplus Im(d)\subseteq Ker(d),$$
from Proposition \ref{pp d^2=0} and we show the other inclusion by induction on the degree of the filtration of $(U_{\epsilon}(\gt)\otimes C_{\epsilon}(V,(\phantom{v},\phantom{v})))^{\gg}$.\\

\noindent
$\bullet$ Initialisation: Since $F_{-1}((U_{\epsilon}(\gt)\otimes C_{\epsilon}(V,(\phantom{v},\phantom{v})))^{\gg})=\lbrace 0 \rbrace$ we clearly have $$Ker(d|_{F_{-1}((U_{\epsilon}(\gt)\otimes C_{\epsilon}(V,(\phantom{v},\phantom{v})))^{\gg})})\subseteq \mathcal{Z}_{\epsilon}(\Delta(\gg))\oplus Im(d).$$
$\bullet$ Induction: Let $n\in \mathbb{N}$ and suppose that
$$Ker(d|_{F_{n}((U_{\epsilon}(\gt)\otimes C_{\epsilon}(V,(\phantom{v},\phantom{v})))^{\gg})})\subseteq \mathcal{Z}_{\epsilon}(\Delta(\gg))\oplus Im(d).$$
Let $x\in Ker(d|_{F_{n+1}((U_{\epsilon}(\gt)\otimes C_{\epsilon}(V,(\phantom{v},\phantom{v})))^{\gg})})$. We have $gr(d)(gr(x))=0$ and by \eqref{proof thm kernel of d 1}, there exist $y\in gr_{n}((U_{\epsilon}(\gt)\otimes C_{\epsilon}(V,(\phantom{v},\phantom{v})))^{\gg})$ and $z\in S_{\epsilon}(\gg)^{\gg}$ such that
$$gr(x)=gr(d)(y)+z\otimes 1.$$
Consider the symmetrisation $s:S_{\epsilon}(\gg)\rightarrow U_{\epsilon}(\gg)$ (see \cite{Scheunert83}). We have $x-d(s(y))-\Delta(s(z))\in Ker(d)$ and
$$x-d(s(y))-\Delta(s(z))\in F_{n}((U_{\epsilon}(\gt)\otimes C_{\epsilon}(V,(\phantom{v},\phantom{v})))^{\gg})$$
since
$$gr(x-d(s(y))-\Delta(s(z)))=gr(x)-gr(d(s(y)))-gr(\Delta(s(z)))=gr(x)-gr(d)(y)-z\otimes 1=0.$$
Hence, by assumption, there exist $y'\in (U_{\epsilon}(\gt)\otimes C_{\epsilon}(V,(\phantom{v},\phantom{v})))^{\gg}$ and $z'\in \mathcal{Z}_{\epsilon}(\Delta(\gg))$ such that
$$x-d(s(y))-\Delta(s(z))=d(y')+z'$$
and so
$$x=d(y'+s(y))+z'+\Delta(s(z))$$
which proves that
$$Ker(d)=\mathcal{Z}_{\epsilon}(\Delta(\gg))\oplus Im(d).$$
\end{proof}
\vspace{0.2cm}

\begin{cor}
Let $z\in \mathcal{Z}_{\epsilon}(\gg)$. There exist a unique $a\in \mathcal{Z}_{\epsilon}(\Delta(\gg))$ and a unique $b \in \Big(U_{\epsilon}(\gt)\otimes C_{\epsilon}(V,(\phantom{v},\phantom{v}))\Big)^{\gg}$ such that
$$z\otimes 1=a+\tilde{D}b+b\tilde{D}.$$
\end{cor}

\begin{proof}
Clearly, $z\otimes 1$ is $\gg$-invariant and in the kernel of $d$. Hence, by Theorem \ref{thm kernel of d}, there exist an unique $a\in \mathcal{Z}_{\epsilon}(\Delta(\gg))$ and an unique $b \in \Big(U_{\epsilon}(\gt)\otimes C_{\epsilon}(V,(\phantom{v},\phantom{v}))\Big)^{\gg}$ such that $z\otimes 1=a+d(b)$. We have $\alpha(z\otimes 1)=0$, then $\alpha(d(b))=0$ and so $\alpha(b)=1$. Hence we have $d(b)=\tilde{D}b+b\tilde{D}$ and we obtain $z\otimes 1=a+\tilde{D}b+b\tilde{D}$.
\end{proof}

\footnotesize
\bibliographystyle{alpha}
\bibliography{CubicDiracoperators}

\end{document}